\documentclass{amsart}
\usepackage[utf8]{inputenc}
\usepackage{times}
\usepackage{geometry}
\usepackage{latexsym, amsmath, amscd,amsthm}
\usepackage{graphicx}
\usepackage[percent]{overpic}
\usepackage{units}
\usepackage{hyperref}
\usepackage{euscript}
\usepackage{multicol}
\usepackage{epstopdf}
\usepackage{paralist}
\usepackage{multicol}
\usepackage{appendix}

\newtheorem{theorem}{Theorem}

\newtheorem{lemma}[theorem]{Lemma}
\newtheorem{proposition}[theorem]{Proposition}
\newtheorem{corollary}[theorem]{Corollary}

\theoremstyle{definition}

\newtheorem{conjecture}[theorem]{Conjecture}

\def\thm#1{Theorem~\ref{thm:#1}}
\def\lem#1{Lemma~\ref{lem:#1}}
\def\fig#1{Figure~\ref{fig:#1}}
\def\prop#1{Proposition~\ref{prop:#1}}
\def\cor#1{Corollary~\ref{cor:#1}}

\newcommand{\R}{\mathbb{R}}

\begin{document}

\title{The Mathematics of Tie Knots}

\author{Elizabeth Denne}
\address{Elizabeth Denne: Washington \& Lee University, Mathematics Department, Lexington VA}
\email[Corresponding author]{dennee@wlu.edu}
\urladdr{https://elizabethdenne.academic.wlu.edu/}
\author{Corinne Joireman}
\address{Corinne Joireman: Washington \& Lee University}
\email{joiremanc21@mail.wlu.edu}
\author{Allison Young}
\address{Allison Young: Washington \& Lee University}
\curraddr{University of Virginia, Charlottesville VA}
\email{aky3gd@virginia.edu}
\date{October 23, 2020}

\begin{abstract} In 2000, Thomas Fink and Young Mao studied neck ties and, with certain assumptions, found 85 different ways to tie a neck tie. They gave a formal language which describes how a tie is made, giving a sequence of moves for each neck tie. The ends of a neck tie can be joined together, which gives a physical model of a mathematical knot that we call a {\em tie knot}. In this paper we classify the knot type of each of Fink and Mao's 85 tie knots. We describe how the unknot, left and right trefoil, twist knots and $(2,p)$ torus knots can be recognized from their sequence of moves. We also view tie knots as a family within the set of all knots. Among other results, we prove that any tie knot is prime and alternating. 
\end{abstract}

\maketitle
\section{Introduction} 
There are many, many ways to tie a neck tie, three such ways are shown below in \fig{examples-tie}. In 2000, Thomas Fink and Yong Mao wrote a fantastic paper \cite{finkmao} where they described 85 ways to tie a neck tie. They later turned this into a delightful little book \cite{finkmao-book}. Many of the 85 neck ties were familiar, but some were new to men's fashion. In their paper, Fink and Mao developed a particular set of rules that describes neck ties. They used a formal language to describe, then count all neck ties. They made several assumptions about neck ties. One of the more important ones is that the facade of the neck tie is flat, as illustrated in the four-in-hand neck tie in \fig{examples-tie} (left). As can be seen in the Van Wijk and Trinity neck ties (center and right in \fig{examples-tie}), such an assumption means that not all neck ties can be described by their language.

\begin{figure}[htbp]
\includegraphics[scale=0.5]{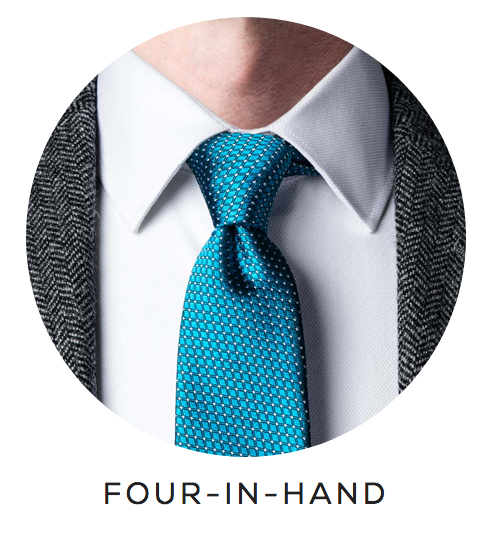}
\includegraphics[scale=0.5]{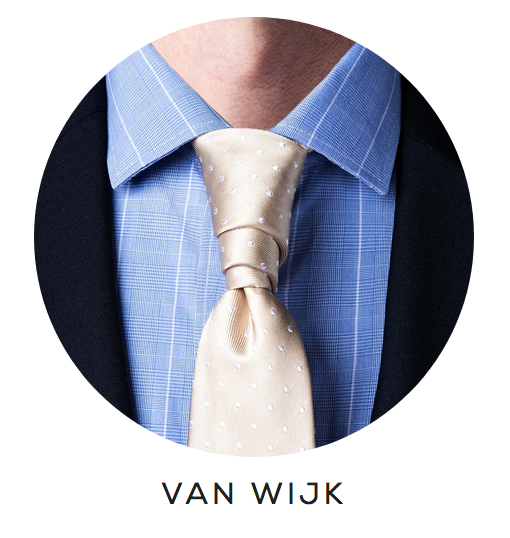}
\includegraphics[scale=0.5]{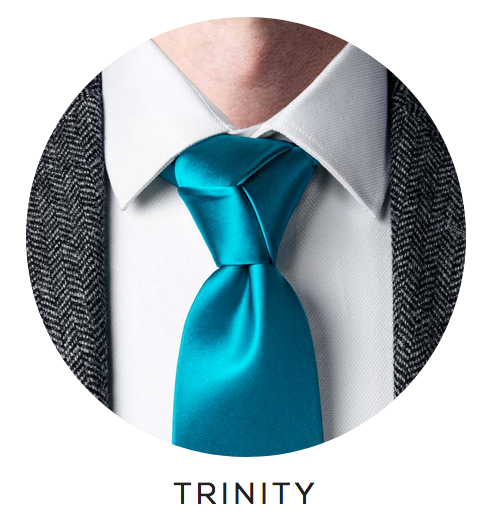}
\label{fig:examples-tie}
\caption{Different examples of tie knots from \cite{Ties}.}
\end{figure}

In 2015, Dan Hirsch, Ingemar Markstr\"om, Meredith L. Patterson, Anders Sandberg and Mikael Vejdemo-Johansson  \cite{more-ties} extended Fink and Mao's enumeration of neck ties to include ties with a textured front, and those tied with the narrow end of a tie. They created a new formal language to do so, and counted $266,682$ tie-knots that seem tie-able.  Vejdemo-Johansson~\cite{generator} hosts an online tie knot generation toy which allows the user to randomly pick one of these neck ties. In addition, Vejdemo-Johansson believes that their count of neck ties is a lower bound. He has a list of known named tie knots and their knotting sequences \cite{known-ties}. There, Vejdemo-Johansson lists the TrueLove tie knot as a knot which does not fit the Fink-Mao description, nor the extended enumeration in \cite{more-ties}. In addition, N. Scoville \cite{scoville} has an interesting webpage where he shows many other interesting neck ties. Many of these do not appear to obey even the more relaxed conditions of \cite{more-ties}. (The Atomic knot, the Danish knot and the Lady slipper knot are a few examples showing extra twists and loops which may not be captured by the formal language of \cite{more-ties}.)

 In this paper, we consider neck ties as mathematical knots, by joining the two ends of the tie together.  In Section~\ref{section:knots}, we first review of some of the basic facts about knots that we need. Next, in Section~\ref{section:tieknots}, we review the Fink and Mao description of neck tie knots, and note that the list of left, right and center moves used to tie a neck tie is called a {\em tie sequence}. We then create a (mathematical) knot simply by joining the two ends of the neck tie to create what we call a {\em tie knot}. Here, we can restrict our attention to the 85 tie knots that Fink and Mao described, we will call these {\em FM-tie knots}. We can also look at the infinite family of tie knots that follow the Fink and Mao rules, but where there are no restrictions on the number of moves needed to create the tie knot.  (Our ties are as long as needed to tie the tie knot.) Thus we consider {\em tie knots} as a family within the set of all knots. Just as we consider torus knots, or alternating knots as a family within the set of all knots.

We worked out the knot type of each of the 85 FM-tie knots. Section~\ref{section:data} summarizes this data; the complete list can be found in Appendix~\ref{appendix:data}. As we worked out the knot type of our tie knots, we realized that the tie sequences can be reduced while preserving the knot type in four standard ways.  These ways are explained in Section~\ref{section:reductions}, along with the result that if a tie knot is tied with $k$ moves, then it has a knot diagram with $k-1$ or fewer moves.  In Section~\ref{section:examples}, we examine the ways the unknot, twist knot and $(2,p)$ torus knots appear as tie knots. We show how to recognize these knots from their tie sequence. As the trefoil knot is both a twist knot and $(2,p)$ torus knot, we discuss this as a separate case.   

Finally, in Section~\ref{section:family}, we prove several more results about tie knots. Perhaps the most important is that any tie knot is prime and alternating. We also show that the diagram corresponding to the reduced tie sequence of a tie knot is itself a reduced knot diagram. There are many other interesting questions still to be answered about tie knots. For example, after looking at the data of tie knots, we conjecture that tie knots are in fact prime, alternating 2-bridge knots.


\section{Knots}\label{section:knots}
We start by recalling some basic facts about mathematical knots that can be found in any introductory knot theory text, for instance \cite{adams,johnhen, liv,mura}. A {\em knot} is a simple closed curve in $\R^3$, and a {\em link} is a disjoint union of knots. Two knots are {\em equivalent} if one can be deformed into the other without one strand of the knot passing through another. (More formally, if they are ambient isotopic.) In this paper, we only consider {\em tame} knots, those that are equivalent to smooth knots. Knots are most often visualized using  {\em{knot diagrams}}, which are knot projections that include the crossing information of the strands. 

\begin{figure}[htbp]
\begin{center}
\includegraphics{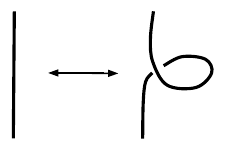} \ \qquad \includegraphics{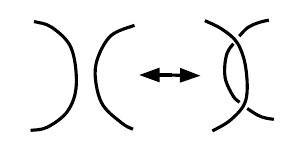}
\ \qquad   \includegraphics[scale=1]{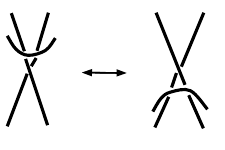}
\caption{Reidemeister moves. From left to right, R1-, R2- and R3-moves.}
\label{fig:R-moves}
\end{center}
\end{figure}

Perhaps an easier way to understand whether two knots are  equivalent is to look at their corresponding knot diagrams. Kurt Reidemeister \cite{Reidemeister} proved that two knots are equivalent if and only if their corresponding knot diagrams are related by planar isotopy and a finite sequence of Reidemeister moves. The
Reidemeister 1, 2 and 3 moves are illustrated in \fig{R-moves}, where we assume that the knot diagram is fixed other than the local changes shown.

In this paper we will often discuss two families of knots: twist knots and torus links.  
A $T_n$ {\em twist knot} is an alternating knot that consists of $n$ crossings created by twisting together two strands and a clasp that creates another two crossings. The $T_3$ twist knot is shown in \fig{exTwistTorus} (left).
A {\em $(p,q)$-torus link}, denoted by $T(p,q)$, is any link which can be embedded in the torus.   A $(p,q)$-torus link twists $p$ times around the longitude and $q$ times around the meridian of the torus. It is a well known fact (see \cite{adams, mura}) that if $\gcd(p,q)=1$, then $T(p,q)$ is a knot. In this paper we focus our attention on the $(2, p)$-torus knots (where $p$ is odd).  For example, Figure~\ref{fig:exTwistTorus} (right) shows a $T(2,5)$ torus knot.

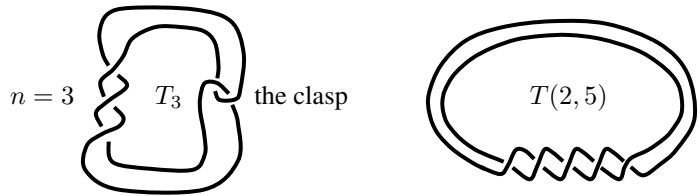
\begin{figure}[htbp]
\begin{center}
\begin{overpic}[scale=1]{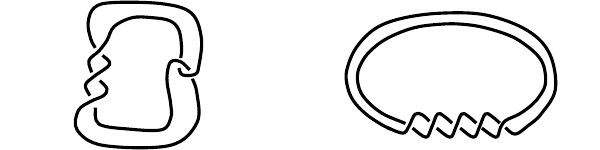}
\put(22,12){$T_3$}
\put(3,12){$n=3$}
\put(35,12){the clasp}
\put(71,12){$T(2,5)$}
\end{overpic}
\caption{On the left, the $T_3$ twist knot is also known as the $5_2$ knot in knot tables. On the right, the $T(2,5)$ torus knot is also known as the $5_1$ knot.}
\label{fig:exTwistTorus}
\end{center}
\end{figure}

\section{Tie Knots}\label{section:tieknots}
In this section we repeat the description and construction of tie knots given by Fink and Mao \cite{finkmao,finkmao-book}.  Readers who know this construction are encouraged to skim through this section to become familiar with the meanings of the emphasized terms.

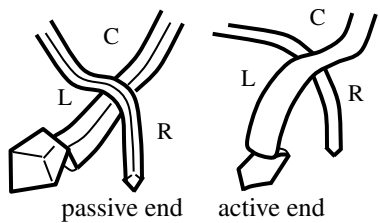
\begin{figure}[htbp]
\begin{center}
\begin{overpic}[scale=1]{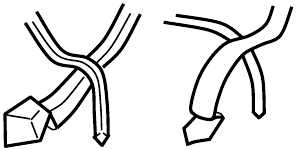}
\put(26,40){\small{C}}
\put(14,24){\small{L}}
\put(40,15){\small{R}}
\put(15,-5){passive end}
\put(80,45){\small{C}}
\put(62,29){\small{L}}
\put(90,25){\small{R}}
\put(56,-5){active end}
\end{overpic}  
\caption{Left, right, and center regions in relation to the first move of a tie.  The left image shows a $L_o$ (out-move) while the right shows a $L_i$ (in-move).}
\label{fig:tieRegions}
\end{center}
\end{figure}

We start by draping a tie around the neck with the thinner end, the {\em passive strand}, to the left. The thicker end is the {\em active strand}, and we assume we start the tie knot by wrapping the active strand to the {\em left} either over or under the passive strand, as illustrated in Figure~\ref{fig:tieRegions}.  We pause to note two things here. First, the images of ties knots in this paper represents the view seen in a mirror. Second, tie knots started by wrapping the active strand to the right (as opposed the left) are mirror images of the left-hand ones,  so we do not consider them here.

Now once the first left move is made, the tie forms a Y-like shape with three regions: left, right, and center (see Figure~\ref{fig:tieRegions}).  A tie knot is then made by half-turns into and out of these regions using the active strand.   In general, when the active strand crosses into a region (and over the passive strand) the move is described as an {\em in-move}, and when the active strand crosses  out of a region (and under the passive strand) the move is an {\em out-move}.  How many moves are used to create a tie knot affects how you start to tie a physical tie. When there are an even number of moves, as in the four-in-hand tie knot,  the ties are begun with fabric lying the correct way around the neck. When there are an odd number of moves, as in the Pratt (or Shelby) tie knot, the underside of the tie is visible around the neck. These choices mean that the correct side of the active strand is visible once the tie knot is completed.

The moves into and out of the three regions gives a finite set of {\em tie moves} $\{ L_i, L_o, R_i, R_o, C_i, C_o\}$, where the subscripts $_i$ or $_o$ represent the direction in or out. A tie knot is thus represented by a {\em tie sequence} which is an ordered list of tie moves. In addition, tie sequences must satisfy the following rules.
\begin{compactenum}
\item[(0)] Tie sequences must begin with $L_i$ or $L_o$.
\item No region ($L,R,C$) is repeated. (For example: $L_iL_o$ is forbidden.) 
\item The direction must alternate between in $_i$ and out $_o$.
\item Tie knots must end with either $L_oR_iC_oT$ or $R_oL_iC_oT$. Here, the $T$ denotes the tuck move, where the end of the tie is threaded through the single loop just created, as illustrated second from the right in \fig{fourinhand}.
\item Ties knots are made with at most 9 tie moves (not counting the final $T$ move).
\end{compactenum}

In the rest of the paper, we will refer to these assumptions as the {\em FM-assumptions}\label{FM-assumptions}. Note that (3) means that when the tie knot is tightened we see a {\em flat facade} with the knot hidden behind it. For example, the four-in-hand tie, shown on the left in \fig{examples-tie}, has such a flat facade. Later on, we will often refer to the fact there are two ways of tying the flat facade. Condition (4) arises naturally as  tie knots are made from a finite length of a tie. Finally, (3) and (4) together mean that all tie sequences have between 3 and 9 tie moves.   With the assumptions (0)--(4), Fink and Mao \cite{finkmao} found 85 distinct tie sequences. As an example, in \fig{fourinhand}, we see the four-in-hand tie knot is made using the tie sequence $L_iR_oL_iC_oT$.  

\begin{figure}[htbp]
\begin{center}
\begin{overpic}[scale=1]{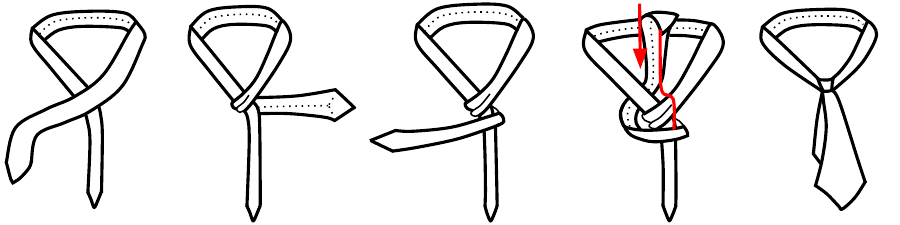}
\end{overpic}  
\caption{The four-in-hand tie knot is formed following tie sequence $L_iR_oL_iC_oT$.}
\label{fig:fourinhand}
\end{center}
\end{figure}

We want to find the knot type of each of the 85 tie sequences. To do this, we first make the tie into a (mathematical) knot by joining the ends of the passive and active strands together, and then ignoring the thickness of the tie, viewing it as a 1-dimensional simple closed space curve. We will call the knots corresponding to the 85 tie sequences {\em FM-tie knots}.  We can also consider the infinite family of {\em tie knots} to be the family of tie knots constructed from tie sequences using FM-assumptions (0)--(3). This means the tie knots can be made with any number of moves that we like\footnote{Thanks to Colin Adams for making this suggestion}. The FM-tie knots are thus a finite subset of the family of tie knots.  We warn the reader that the word tie knot can refer to both a  physical knot made in a tie from a tie sequence, and also to refer to the (mathematical) knot created from it. The context will make it clear which is meant.

We add one more remark about tie sequences (following \cite{finkmao}). Since each tie sequence ends in $C_0$, then by moving backwards along the tie sequences, we can deduce that tie sequences with an even number of moves start with $L_i$ while tie sequences with an odd number of moves starts with $L_o$. However, we will usually include the $_i$ and $_o$ subscripts to make the tie sequences easy to follow. 

We conclude this section by remarking that in their paper,  Fink and Mao \cite{finkmao}  also discuss some other aesthetic considerations when tying ties.   Some of the 85 tie sequences corresponded to known tie knots, but many were new. This is explored further in their book \cite{finkmao-book} and they note their list is not exhaustive by any means. For example, tie knots like Van Wijk \cite{VanWijk} shown in the center in \fig{examples-tie} (or the Prince Albert \cite{Albert}) are completed by the active strand tucking under 3 (or 2 loops) of the tie. These tie knots do not satisfy condition (3) above and so are not found in Fink and Mao's list. As discussed in the introduction, these ties, and many others are found in \cite{more-ties}.


\subsection{Tie knot diagrams}\label{tiediagrams}

To understand tie knots, we first create knot diagrams to represent them. We will refer to these as  {\em tie knot diagrams}, or more usually, {\em tie diagrams}. There is a standard way to do this as shown in \fig{tie-diagram}. Note that we divide each tie knot into two pieces. The active strand starts at point $a$ (just before the first left crossing), follows the tie sequence around the tie and ends at point $b$. The passive strand which starts at point $b$ travels up the {\em straight strand}, then around the {\em neck loop} back to point $a$. We will usually include the points $a$ and $b$ in our tie diagrams to make them easier to read. Tie knots can by given a natural orientation by following the order of the tie sequence along the active strand.
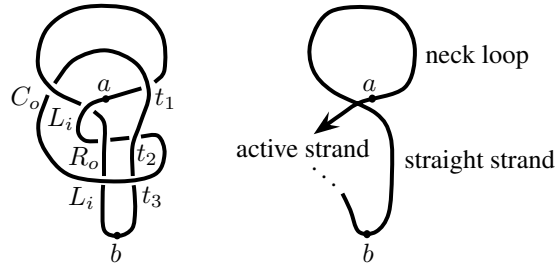
\begin{figure}[htbp]
\begin{center}
\begin{overpic}[scale=1]{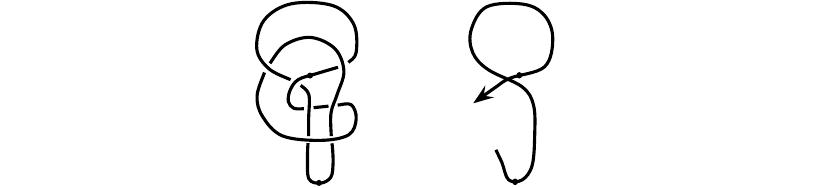}
\put(32,11){$L_i$}
\put(34,7.3){$R_o$}
\put(34,3.5){$L_i$}
\put(28.5,13){$C_o$}
\put(36.75,14.3){$a$}
\put(38,-2){$b$}
\put(42.25,13){$t_1$}
\put(40.5,7.7){$t_2$}
\put(41,3.7){$t_3$}
\put(62,14.3){$a$}
\put(62,-2){$b$}
\put(68,17){neck loop}
\put(66,7){straight strand}
\put(50,8){active strand}
\put(57,5){$\ddots$}
\end{overpic}  
\caption{On the left, the tie diagram for the four-in-hand tie knot $L_iR_oL_iC_oT$. On the right, the passive strand consists of the neck loop and the straight strand, while the active strand (not shown) starts at point $a$ and follows on from the arrow to point $b$.}
\label{fig:tie-diagram}
\end{center}
\end{figure}

When we create our tie knot diagrams, we assume that there are no unnecessary crossings. This means crossings appear in precisely two ways. The first way is from $L,R,C$ moves, where the active strand crosses over or under the straight strand or the neck loop of the passive strand. An in-move means the active strand crosses over the passive strand, and an out-move means the active strand crosses under the passive strand.  The second way crossings appear is from the final $T$ tuck move. This consists of three crossings which we denote by $t_1,t_2,t_3$. The $t_1$ crossing occurs when the active strand crosses the neck loop,  and so $t_1=L_i$ or $t_1=R_i$ (whichever makes the most sense). This is shown in \fig{tie-diagram} left. The $t_2$ and $t_3$ crossings are formed when the active strand pases over, then under the loop of the active strand. So the $t_2$ and $t_3$ crossings are the only crossings of the active strand with itself.

The assumptions that there are no unnecessary crossings has two consequences for the crossings corresponding to the $L,R,C$ moves. As we move along the active strand we first note that the crossings appear in a downwards direction along the straight part of the passive strand. Secondly, as we move along the active strand, the crossings along the neck loop move in either an increasing clockwise direction from the start point $a$ on the left side, or in an increasing counterclockwise direction from point $a$ on the right side.  

In the rest of the paper, we will use tie sequence notation to indicate crossings in the corresponding tie diagram, remembering that the directions $_i$ or $_o$ correspond to over- and under crossings respectively. We use the notation [\dots] to indicate some sequence of tie knot moves. We immediately deduce the following useful lemma. 

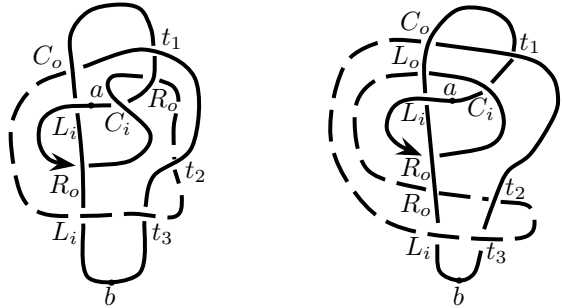
\begin{figure}[htbp]
\begin{overpic}{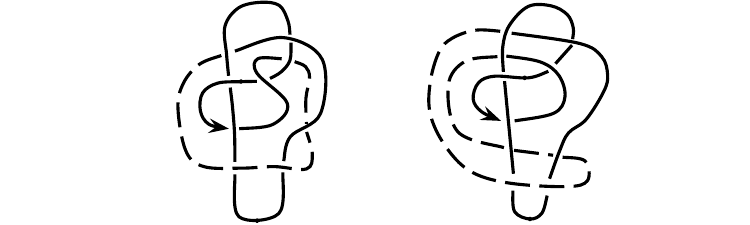}
\put(28,16){$L_i$}
\put(27.7,10){$R_o$}
\put(33.5,16.5){$C_i$}
\put(38,19){$R_o$}
\put(28,5){$L_i$}
\put(26,23.5){$C_o$}
\put(39.5,25.5 ){$t_1$}
\put(42,11.5 ){$t_2$}
\put(38.5,5 ){$t_3$}
\put(32,20){$a$}
\put(33.5,-1.75){$b$}
\put(64.5,17.5){$L_i$}
\put(64.5,11.75){$R_o$}
\put(71.5,18){$C_i$}
\put(63.5,23.5){$L_o$}
\put(64.5,8){$R_o$}
\put(65,3.5){$L_i$}
\put(64.5,27){$C_o$}
\put(77,25 ){$t_1$}
\put(75.5,9.75 ){$t_2$}
\put(73.5,3 ){$t_3$}
\put(69,20.25){$a$}
\put(70,-1.75){$b$}
\end{overpic} 
\caption{In the tie sequence $L_iR_oC_iR_oL_iC_oT$, the $C_iR_o$ crossings on the left are replaced by $C_iL_oR_o$ on the right. The dotted line shows the strand of the knot which moves.}
\label{fig:X-move}
\end{figure}

\begin{lemma}\label{lem:crossing-replace} Any tie knot diagram with crossings $[\dots] CL [\dots]$ can be replaced by $[\dots] CRL[\dots]$, where both the $RL$ crossing type match $L$, without changing the knot type. Similarly,
any tie knot diagram with crossings $[\dots] CR[\dots]$ can be replaced by $[\dots] CLR[\dots]$, where both the $LR$ crossing type match $R$, without changing the knot type.
\end{lemma}

This is illustrated in Figure~\ref{fig:X-move}. Before the proof, we note two things. First, we have not listed the direction (in or out) of each tie move so as to not overwhelm the reader with cases. Second, we cannot use the replaced crossings to give a tie sequence satisfying the FM-assumptions, since condition (2) fails. 

\begin{proof} The case  $[\dots] CR[\dots]$ to $[\dots] CLR[\dots]$ is shown in Figure~\ref{fig:X-move}.  Recall that the start of the tie sequence starts just after point $a$. The two knots are equivalent as can be seen by moving the strand under (or over) the rest of the knot to the new position.  
\end{proof}


\section{Knot type data}\label{section:data}
We took each of Fink and Mao's 85 tie sequences, created the corresponding tie knot, and found the knot type\footnote{Fink and Mao \cite{finkmao} listed the knot type for some tie sequences, but not all 85.}. A tedious, but straightforward check verifies each knot type.  For small crossing knots, we classified the knot type by hand using knot tables (see for instance \cite{adams, liv, Rolf}). For the higher crossing knots (that were more complex), we used the Knot Identification Tool \cite{KIT}.
A complete list of the tie knots and their corresponding knot type can be found in Appendix~\ref{appendix:data}. The list is ordered by knot type, then by the Fink and Mao numbering (1--85). We have summarized the data from the list below.

We found 27 different knot types. Table~\ref{tab:table1} lists all of the knot types that appear for a particuar number of tie moves. As expected, as the number of tie moves increases, the number of crossings in the tie knot diagrams increases, and so the number of knot types increases.
We also see that a particular knot type can be made with different numbers of tie moves. For example, the $3_1$ trefoil knot can be made from tie sequences with 4, 7, 8, or  9 tie moves. 

\begin{table}[htbp]
 \begin{tabular}{| c|l |}
 \hline
    \# Tie Moves & Knot Types  \\ \hline \hline
    3 & $0_1$\\
    4 & $3_1$\\
    5 & $0_1,4_1$\\
    6 & $0_1, 5_1, 5_2$\\
    7 & $0_1,3_1, 6_1,6_2,6_3$\\
    8 & $0_1,3_1, 4_1, 7_1, 7_2,7_3, 7_4,7_5, 7_6, 7_7$\\
    9 & $0_1,3_1, 4_1, 5_1, 5_2,  8_1,8_2, 8_3, 8_4, 8_6, 8_7, 8_8, 8_9, 8_{11}, 8_{12}, 8_{13}, 8_{14}$\\
    \hline
\end{tabular}
\caption{Knot types occurring for each number of tie moves.}
\label{tab:table1}
\end{table}

By further examining the data in Appendix~\ref{appendix:data}, we see that of the 85 tie knots, there are
\begin{compactenum}
\item11 unknots, 
\item 13 trefoil knots, 
\item 25 additional twist knots ($4_1, 5_2, 6_1, 7_2, 8_1$),
\item 6 additional $(2,p)$ torus knots ($5_1,7_1$).
\end{compactenum}
We list the trefoil knot separately, since it is both a twist and torus knot.  This means that about 13\% of the tie knots are unknots, and about 52\% are trefoil, twist or $(2,p)$ torus knots. The remaining 35\% (or 30 out of 85) tie knots represent the remaining 18 knot types.

Another way to summarize these numbers is to look at Tables~\ref{tab:table2} and~\ref{tab:table3}. In the Table~\ref{tab:table2}, the knot type is given in the first column, and the number of tie knots for that knot type is given in the second column. Now, for a particular knot type, the third column shows the number of tie knots for each number of tie moves. For example, the second row shows that there are 13 tie knots which are trefoil knots.  The third column shows that of those 13 tie knots, one has 4 moves, four have 7 moves, four have 8 moves and four have 9 moves.

\begin{table}[htbp]
 \begin{tabular}{| c | c | l |}
 \hline
   Knot type & \# Tie knots & \# Tie knots $\times$ \# Tie moves.  \\ \hline \hline
  $0_1$ & 11 & $1\times3, 2\times5 , 2\times 6, 2\times 7, 2\times  8, 2\times 9$\\ \hline
  $3_1$ & 13 & $1\times 4, 4\times 7, 4\times 8, 4\times 9$\\ \hline
$4_1$ & 9 & $ 1\times 5, 4\times 8, 4\times 9$ \\ \hline
$5_1$ & 5 & $ 1\times 6, 4\times 9$ \\ \hline
$5_2$ & 10 & $2\times 6, 8\times 9 $ \\ \hline
$6_1$ & 2 & $2\times 7$\\ \hline
$7_1$ & 1 & $1\times 8$ \\ \hline
$7_2$ & 2 & $2\times 8$ \\ \hline
$8_1$ & 2 & $2\times 9$ \\ \hline
\end{tabular}
\caption{Number of tie knots of a given knot type for a particular number of tie moves.}
\label{tab:table2}
\end{table}

The remaining knot types either have 1 or 2 tie knots. Rather than continue the table above, we have condensed this information in Table~\ref{tab:table3}. We have not listed the number of tie moves, since the knots with crossing number 6 all come from tie knots with 7 tie moves. Similarly, the knots with crossing number 7, respectively 8, come from tie knots with 8, respectively 9, tie moves. 

\begin{table}[htbp]
 \begin{tabular}{| c| l| }
 \hline
    \# Tie knots & Knot Type \\ \hline \hline
2 & $6_2, 7_3, 7_5, 7_6, 8_2, 8_4,8_6, 8_7, 8_8,  8_{11}, 8_{13}, 8_{14}$\\ \hline
1 & $6_3,7_4,  7_7, 8_3, 8_9, 8_{12}$
\\ \hline
\end{tabular}
\caption{Number of tie knots for a particular knot type}
\label{tab:table3}
\end{table}

The reader will have observed that many knots can be tied with 2 or 4 possibilities for a particular number of moves. This is often because there are two ways to finish the tie sequence ($L_oR_iC_oT$ or $R_oL_iC_oT$), and for those knots the knot type is determined by crossings earlier in the tie sequence. On the other hand, some knots ($7_1,6_3,7_4,7_7,8_3,8_9,8_{12}$) have just one tie sequence, as the tie sequence is minimal in some sense.  We will explore all of these ideas further in the next section.

\section{Reducing tie knot sequences}\label{section:reductions}

When any tie sequence is made into a tie knot, there are several natural ways to reduce the number of crossings in the tie knot diagram while maintaining the knot type. An example is shown in \fig{Red-I}. In this section we define these reduction moves, and then apply them to understand some of the underlying structures we noticed about tie knots in the previous section. We finish by discussing how the crossing number of a tie knot is related to the number of tie moves, and observing that all 85 FM-tie knots are alternating.

Recall that we use the tie sequence to describe the crossings in the tie diagram. Thus, we will define the reduction moves in terms of the tie sequence. Fink and Mao~\cite{finkmao} also list some reduction moves, but they did not carefully consider the crossings arising from the tuck move. We have clarified this point and given slightly different reduction moves here.   Recall that in the tie knot diagram the tuck $T$ move corresponds to three crossings $t_1,t_2,t_3$, that $[\dots]$ represents part of a tie sequence, and we include the start and end points $a$ and $b$ of the active strand in our figures.  In what follows, if the direction (in or out) of the tie moves is not listed, then any choice of direction following the FM-assumptions is acceptable.

\begin{figure}[htbp]
\begin{center}
\begin{overpic}[scale=1]{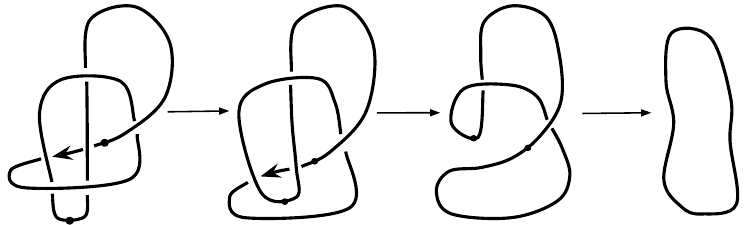}
\put(5,-2){A}
\put(8,11 ){$L_o$}
\put(8.5,6){$R_i$}
\put(19,11){$C_o$}
\put(9,21){$t_1$}
\put(3.2,9.75){$t_2$}
\put(4.25,2.75){$t_3$}
\put(14,12){$a$}
\put(9.5,-2){$b$}
\put(40,-2){B}
\put(36,8.5){$L_o$}
\put(46.5,10){$C_o$}
\put(36.5,20.5){$t_1$}
\put(31,6.5){$t_2$}
\put(41,9.5){$a$}
\put(39,1.5){$b$}
\put(65,-2){C}
\put(74.5,12.5){$C_o$}
\put(62,19.75){$t_1$}
\put(71.5,9.5){$a$}
\put(63,9){$b$}
\put(91,-2){D}
\end{overpic}
\caption{Moving from step A to step D, we see two R2-moves and two R1-moves takes $L_oR_iC_oT$  and reduces it to $\emptyset$, the unknot.}
\label{fig:Red-I}
\end{center}
\end{figure}
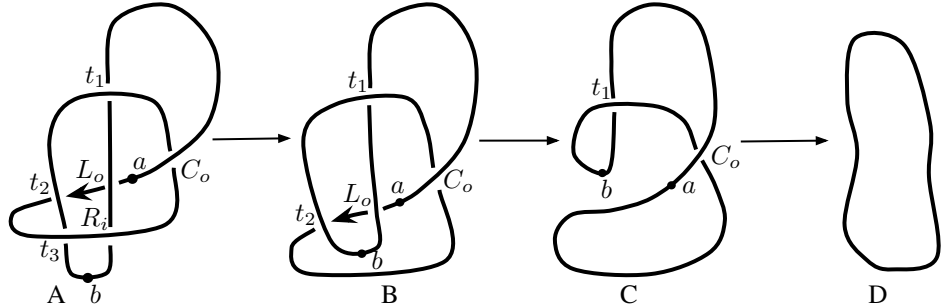

\begin{lemma}[Reduction Lemma] The following reductions in a tie sequence do not change the knot type:
\begin{enumerate}
\item[(0)] $L_oR_iC_oT$ reduces to $\emptyset$ (the unknot),
\item[(I)]  $[\dots] L_iR_oL_iC_oT$ reduces to $[\dots] L_iC_ot_1$, and  
$[\dots] R_iL_oR_iC_oT$  reduces to $[\dots] R_iC_ot_1$,
\item[(II)]  $[\dots] C_iR_oL_iC_oT$ and $[\dots] C_iL_oR_iC_oT$ both reduce to $[\dots]$,
\item[(III)]  $f(L,R)$ reduces to $\emptyset$,
\item[(IV)]  $[\dots]Cf(L,R)$ reduces to $[\dots]C\ast$.
\end{enumerate}
Here, $f(L,R)$ represents a sequence of alternating left and right moves (starting in any order), and  $\ast$ represents the first $L$ or $R$ move from $f(L,R)$.
\label{lem:reduction}
\end{lemma}

Note that the tie sequences found in reductions III and IV only appear after reduction II has been applied to a tie sequence.

\begin{proof} The knot type does not change for any of these reductions, as they can be achieved by using Redeimeister moves on the tie diagram. Reduction 0 corresponds to untucking the base of the tie using two R2-moves, then two R1-moves untwist the $C_o$ and $t_1$ crossings, and is shown in \fig{Red-I}.  For the tie sequences in reduction I, observe that the final loop the tuck passes through is formed by the final two L and R tie moves crossing the straight strand of the tie.  We can re-use \fig{Red-I} steps A to C to illustrate this, with the reduction I move given by two R2-moves. (A reduction I move is also shown on the left in \fig{trefoils-examples}.)

The reduction II move is a little more complicated, since the final loop the tuck passes through is formed when the final two L and R tie moves cross the neck loop then straight strand of the tie. This is shown on the left in Figure~\ref{fig:X-move}. We first make the crossing change described by \lem{crossing-replace}, giving the knot in on the right in Figure~\ref{fig:X-move} and also in \fig{Red-II} step A. After that, two R2-moves are used to go from steps A to B in \fig{Red-II}, then two more R2-moves are used to move from steps B to D. 

Reduction III is shown in \fig{Red-II} step D and follows from R1-moves. Reduction IV corrects a small error in \cite{finkmao} (Equation 19 at the end of Section 9). Reduction IV follows by using R1-moves, and is illustrated in \fig{Red-III} step A. 
\end{proof}

\begin{figure}[htbp]
\begin{center}
\begin{overpic}[scale=1]{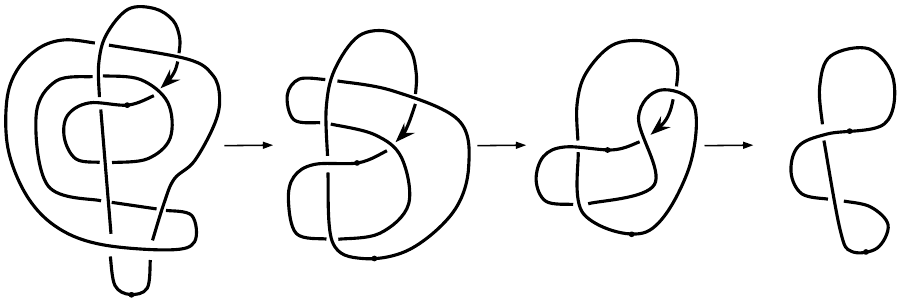}
\put(7, 1){A}
\put(13.75,22.5){$a$}
\put(8.5,19.5){$L_i$}
\put(8.5,16){$R_o$}
\put(16,20.5){$C_i$}
\put(8,25.75){$L_o$}
\put(9,12){$R_o$}
\put(9.5,7){$L_i$}
\put(8.5,29.5){$C_o$}
\put(20.5,27.5){$t_1$}
\put(19,10.75){$t_2$}
\put(17.5,3.75){$t_3$}
\put(13.5,-1.5){$b$}
\put(37, 1){B}
\put(39,16){$a$}
\put(34,13){$L_i$}
\put(33.7,7.75){$R_o$}
\put(45,16.25){$C_i$}
\put(33.5,20.5){$L_o$}
\put(34, 25.5){$C_o$}
\put(47,23){$t_1$}
\put(41,2.5){$b$}
\put(65, 3){C}
\put(67,17.5){$a$}
\put(61.5,15){$L_i$}
\put(61,11.5){$R_o$}
\put(73,16.5){$C_i$}
\put(76,24){$t_1$}
\put(69,5){$b$}
\put(90,3){D}
\put(94,19.5){$a$}
\put(88.5,19){$L_i$}
\put(89.5,12){$R_o$}
\put(96,3){$b$}
\end{overpic}
\caption{Prior to step A, the tie diagram of $L_iR_oC_iR_oL_iC_oT$  is changed to $L_iR_oC_iL_oR_oL_iC_ot_1t_2t_3$ (using \lem{crossing-replace} in Figure~\ref{fig:X-move}). Reduction II takes this to $L_iR_o$ in step D. Finally reduction III takes this to the unknot (not shown).}
\label{fig:Red-II}
\end{center}
\end{figure}
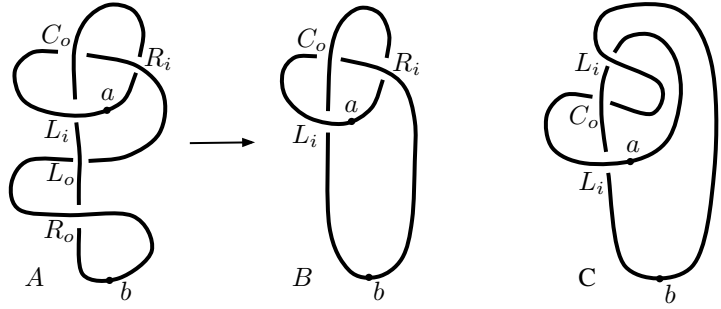
\begin{figure}[htbp]
\begin{center}
\begin{overpic}[scale=1]{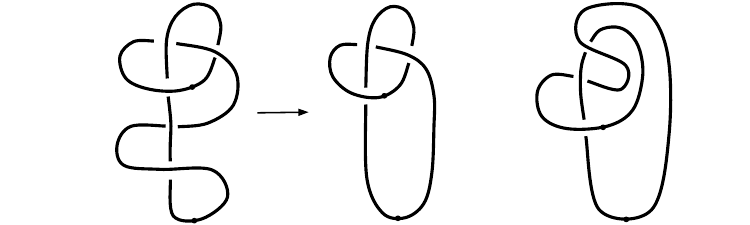}
\put(17, 0){$A$}
\put(25,19.5){$a$}
\put(19, 15.5){$L_i$}
\put(18.75, 25.5){$C_o$}
\put(29.5, 23){$R_i$}
\put(19.25, 11){$L_o$}
\put(19, 5){$R_o$}
\put(27,-1.5){$b$}
\put(45,0){$B$}
\put(50.5,18.25){$a$}
\put(45, 15){$L_i$}
\put(45.5, 25){$C_o$}
\put(55.5, 22.5){$R_i$}
\put(53.5,-1.5){$b$}
\put(75,0){C}
\put(80,14){$a$}
\put(75, 10){$L_i$}
\put(74, 17){$C_o$}
\put(74.55, 22.5){$L_i$}
\put(84,-1.75){$b$}
\end{overpic}
\caption{Reduction IV is applied to reduced tie sequence $L_i C_o R_i L_oR_i$ in step A to get $L_i C_o R_i$ in step B. Steps B and C shows that tie sequences ending  $[\dots]C_oR_i$, or  $[\dots]C_oL_i$ represent the same knot.}
\label{fig:Red-III}
\end{center}
\end{figure}

After applying the Reduction Lemma (\lem{reduction}) to any tie sequence, we obtain the following corollary about {\em reduced tie sequences}. 

\begin{corollary}\label{cor:reduction}
All tie sequences reduce to either
$L[\dots ]C\ast$, $L[\dots ]Ct_1$, or $\emptyset$. Here $\ast$ and $t_1$ are either $L$ or $R$ moves.
\label{cor:Form}
\end{corollary}
\begin{proof} Apply Reductions (0)--(IV) to any tie sequence.
\end{proof}

The following is immediately clear.

\begin{lemma} \label{lem:end-same}
Suppose we have reduced tie sequences matching up until the last tie-move. Then tie sequences ending with $[\dots]C_ot_1$, $[\dots]C_oL_i$ and $[\dots]C_oR_i$ represent equivalent knots. Tie sequences ending $[\dots]C_iL_o$ and $C_iR_o$ also represent equivalent knots. 
\end{lemma}
\begin{proof} Recall that $t_1$ corresponds to $L_i$ or $R_i$ (whichever works best). The tie knots are all equivalent, since active strand can leave the neck loop from either the left or the right side, as can be seen by a simple planar isotopy. This is illustrated in \fig{Red-III} steps B and C.
\end{proof}

The last tie move in a reduced tie sequence is the only tie move that can change without changing the knot type.  If an earlier move is changed then the knot type may change as well.  For example, the following tie sequences differ in the third move: $L_iC_oR_iC_oL_iR_oL_iC_oT$ is a $7_5$ knot but $L_iC_oL_iC_oL_iR_oL_iC_oT$ is a $7_1$ knot (See Appendix~\ref{appendix:data} tie knots 38 and 42).  

We can now give a better explanation of the data in Section~\ref{section:data}. We saw that many knots could be tied using just one tie sequence ($7_1,6_3,7_4,7_7,8_3,8_9,8_{12}$).  This is explained by noting these ties only use reduction~I. After this is performed, the tie sequence cannot be further reduced, and for these examples, the knot is represented by a minimal crossing diagram.  We also observed that many knots can be tied in 2 or 4 different ways for a particular number of tie moves.  This can also be explained using reductions. Recall that there are two ways to finish the tie sequence: $L_oR_iC_oT$ or $R_oL_iC_oT$. 
 If we  restrict our attending to tie knots ending in $[\dots]C_iL_oR_iC_oT$ or $[\dots]C_iR_oL_iC_oT$ and apply reduction II, then these tie sequences reduce to $[\dots]$, and hence we get the same tie knot.  At this point we may also be able to apply reduction III or IV. With these reductions, there are two possibilities for the tie sequence of $L$ and $R$ moves (depending on where you start). There are thus four choices of ending for the tie knot made from enough tie moves (two for choice $L$, $R$ sequence and two choice of type of facade).  We will see these ideas in action in Section~\ref{section:examples}  when we discuss the different tie sequences for the unknot, trefoil, twist and $(2,p)$ torus knots. 
Following from \cor{reduction} we can deduce the following results about crossing number of tie knots.

\begin{proposition}\label{prop:numCross}
For all tie knots with $k$ moves in their tie sequence, there is a knot diagram with $k-1$ or fewer crossings.
\end{proposition}
\begin{proof}
Assume we have a tie sequence with $k$ moves. We know this sequence must end in one of four ways. The first two ways are $[\dots ]R_o L_i C_oT$ or $[\dots ]L_o R_i C_oT$. Remembering that $T$ move corresponds to three crossings $t_1,t_2,t_3$, we get  a tie knot diagram with $k+3$ crossings. Using reduction I in \lem{reduction}, we see the tie sequences can be reduced to $[\dots ]C_o t_1$, removing four crossings. This means that the knot diagram has $k-1$ crossings. If the tie sequence ends in $[\dots]C_iL_oR_iC_oT$ or $[\dots]C_iR_oL_iC_oT$, then applying reduction II means the tie sequence reduces to $[\dots]$ and the corresponding knot diagram has $k-4$ crossings.
\end{proof}

\begin{corollary} \label{cor:numCross}
A tie knot with $k$ moves in its tie sequence has crossing number less than or equal to $k-1$.
\end{corollary}
\begin{proof} Since there is a corresponding knot diagram with $k-1$ or fewer crossings, the crossing number must be less than or equal to $k-1$.
\end{proof}

In addition to these results, we can immediately show the following.
\begin{proposition}
FM-tie knots are alternating.
\label{prop:alternating}
\end{proposition}

\begin{proof}
The first non-alternating knot in the knot table found in the Appendix in \cite{adams} has crossing number 8 (it is $8_{19}$). By Corollary~\ref{cor:numCross}, the only way to get a knot with crossing number 8 is a tie sequence with 9 tie moves. However, none of the tie sequences with 9 tie moves  in Appendix~\ref{appendix:data} corresponds to a non-alternating knot. 
\end{proof}

In Section~\ref{section:family} we generalize this result and prove that any tie knot (of any length) is alternating. We also show that the tie diagrams corresponding to reduced tie sequences are reduced knot diagrams. This allows us to conclude in Corollary~\ref{cor:min-crossing} that any tie knot has a minimal crossing diagram corresponding to its reduced tie sequence.

\section{Unknot, trefoil, torus and twist knots}\label{section:examples}
By examining the tie knot sequences corresponding to a particular knot type, we can make some deductions about their form. In particular, we can learn to recognize a knot type from a tie sequence. In this section, we will refer to all of the FM-tie knots that are unknots as {\em FM-tie unknots}. We will similarly refer to FM-tie trefoil knots, FM-tie twist knots and FM-tie $(2,p)$ torus knots.

\subsection{The unknot}
The unknot $L_iR_oC_iR_oL_iC_oT$ shown in Figures~\ref{fig:X-move} and \ref{fig:Red-II} is a typical example of the unknots that appear in our data.

\begin{lemma} \label{lem:unknot}
If a tie sequence is one of the following types: 
\begin{compactitem}
\item $L_oR_iC_oT$,
\item $g(L,R)C_i R_o L_i C_o T$,
\item $g(L,R)C_i L_o R_i C_o T$,
\end{compactitem}
then the corresponding tie knot is an unknot. Here $g(L,R)$ is an alternating sequence of left and right moves starting with a left move. In addition, all of the FM-tie unknots have one of these three tie sequences.
\end{lemma}
\begin{proof} The tie sequence $L_oR_iC_oT$ reduces to $\emptyset$ by reduction 0, as shown previously in Figure~\ref{fig:Red-I}. We know tie sequences of the second and third form are reduced to $\emptyset$ by using reductions II and III. Thus the corresponding tie knots are all unknots. 

By examining all the FM-tie unknots in Appendix~\ref{appendix:data}, we can see that the only way unknots appear is when their tie sequence is one of the three listed above.
\end{proof}

This lemma also clarifies why we see the  first line in Table~\ref{tab:table2}. Aside from the single unknot made from 3 tie moves, the the other unknots have 5 or more tie moves. There are two unknots for each number of moves, corresponding to the two different ways to tie a flat facade.

Examples of known tie knots that are unknots include: Oriental, Nicky, Pratt, half-Windsor, and St. Andrew tie knots.


\subsection{The trefoil knot}
\begin{lemma}\label{lem:trefoil} If a tie sequence reduces to $LC\ast$ (where $\ast$ is a $L$ or $R$ move), then the corresponding tie knot is a trefoil knot. In addition, all of the FM-tie trefoil knots have a reduced tie sequence of the form $LC\ast$.
\end{lemma}

\begin{proof} If the reduced tie sequence is $LC\ast$ then, by FM-assumptions, there are three alternating crossings in the tie diagram. The arrangement of the crossings shows that the knot must be a trefoil knot.

By examining the tie sequences in Appendix~\ref{appendix:data} corresponding to the trefoil knot, we see there are only three kinds:
\begin{compactitem}
\item $L_iR_oL_iC_oT$,
\item $LCf(LR)C_i R_o L_i C_o T$,
\item $LCf(L,R)C_i L_o R_i C_o T$,
\end{compactitem}
where $f(L,R)$ is a sequence of alternating $L$ and $R$ moves starting in any order. The last two tie sequences start with $L_i$ if there an even number of tie moves in $f(L,R)$, and they start with $L_o$ if there are an odd number.

Now $L_iR_oL_iC_oT$ reduces to $L_iC_ot_1$ by reduction I, as  shown in \fig{trefoils-examples} left. (Recall that $t_1$ is either a $L$ or $R$ move.) The other two tie sequences  both reduce to $LCf(L,R)$ by reduction II. Reduction IV then takes these to $LC\ast$.
\end{proof}

\begin{figure}[htbp]
\begin{center}
\begin{overpic}[scale=1]{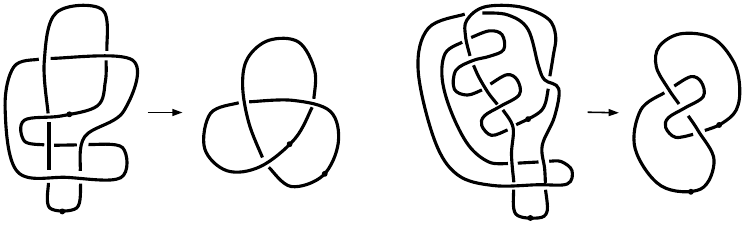}
\put(8.5,15.5){$a$}
\put(3.25,15.25){$L_i$}
\put(2.25,8){$R_o$}
\put(3,3.5){$L_i$}
\put(1.75,23){$C_o$}
\put(14.75,23.5){$t_1$}
\put(9,11.75){\small{$t_2$}}
\put(11.25,7){$t_3$}
\put(7.5,-1.25){$b$}
\put(36,11){$a$}
\put(34,5){$L_i$}
\put(28.25,17.25){$C_o$}
\put(42.5,16.5){$t_1$}
\put(43.75,5){$b$}
\put(69.5,11.75){\small$a$}
\put(65.25, 9.75){\scriptsize{$L_o$}}
\put(68, 14.75){\scriptsize{$C_i$}}
\put(62, 15){\footnotesize{$L_o$}}
\put(63.85, 20){\footnotesize{$C_i$}}
\put(63, 23){\scriptsize{$L_0$}}
\put(65,6){\scriptsize$R_o$}
\put(65.5,2.5){\small$L_i$}
\put(59.5,28.75){\small$C_o$}
\put(74,19.5){\small$t_1$}
\put(72.75,9.5){\small$t_2$}
\put(73.25,2.75){\small$t_3$}
\put(70,-2){$b$}
\put(95,11){\small$a$}
\put(91.5,9){\small{$L_o$}}
\put(92.4,14.5){\footnotesize{$C_i$}}
\put(85,18){\small{$L_o$}}
\put(91,1.5){$b$}
\end{overpic}
\caption{The left two pictures shows the tie knot $L_iR_oL_iC_oT$ is a right handed trefoil knot. The right two pictures shows the tie knot $L_oC_iL_oC_iR_oL_iC_oT$ is a left handed trefoil knot. Here, $L_oC_iL_oC_iR_oL_iC_oT$ was first changed to $L_oC_iL_oC_iL_oR_oL_iC_oT$ using \lem{crossing-replace}.  }
\label{fig:trefoils-examples}
\end{center}
\end{figure}
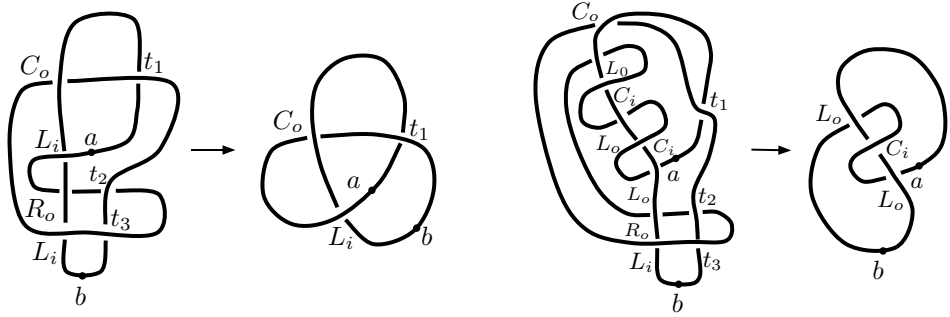

\begin{corollary} Both left and right handed trefoil knots are tie knots.
\end{corollary}
\begin{proof} Right handed trefoil knots are given by the tie sequences in \lem{trefoil} starting with $L_i$. (That is both $L_iR_oL_iC_oT$, and the other two tie sequenes with an even number of $f(L,R)$ moves.) Left handed trefoil knots are given by tie sequences in \lem{trefoil} starting with $L_o$, that is those with an odd number of $f(L,R)$ moves.
\end{proof}

In \fig{trefoils-examples}, we see both left and right haded trefoil knots shown as a tie knot and then as reduced to an minimal-crossing, alternating knot diagram.  This discussion also allow us to fully understand the numbers of trefoil knots that we see in line 2 of Table~\ref{tab:table2}. There is just one trefoil knot with 4 tie moves. Since we are using reductions II and IV, the rest of the trefoils have 7 or more tie moves. Moreover, there are 4 tie sequences for each number of tie moves. (Two for the choice of start move for $f(L,R)$ sequence, then two for the choice of ways of tying the flat facade.)

Examples of known tie knots that are trefoil knots include: four-in-hand, Plattsburgh, and Windsor tie knots. Of these three, the Plattsburgh is left-handed, the other two are right-handed.

\subsection{Twist knots}

A twist knot $T_n$ is characterized by $n$ twists and a clasp, as shown in \fig{exTwistTorus}. By examining the tie sequences, we can see three distinct ways twist knots can be created.

\begin{lemma}\label{lem:twist}
If a tie sequence reduces to one of the following types:
\begin{compactenum}
\item $g(L,R)C_0\ast$, 
\item $LRC(RCRC\dots RC)\ast$, 
\item $LC(RCRC\dots RC)\ast$,
\end{compactenum}
then the corresponding tie knot is a twist knot. 
Here, $\ast$  is a $L$ or $R$ move, $g(L,R)$ is a sequence of alternating $L$ and $R$ moves starting with a $L$ move, and $RCRC\dots RC$ is a sequence of an even number of alternating $R$ and $C$ moves starting with a $R$ move. In addition, all of the FM-tie twist knots have one of these reduced tie sequences.
\end{lemma}

\begin{proof} By examining the tie sequences of twist knots in Appendix~\ref{appendix:data}, we see that the reduced tie sequences listed above are the only ones that appear. Type~1 occurs in all the twist knot types: $3_1, 4_1, 5_2, 6_1, 7_2$ and $8_1$. Type 2 occurs only in twist knots with an even number of twists: $4_1, 6_1$ and $8_1.$
Type 3 occurs only in twists knots with an odd number of twists: $3_1, 5_2$ and~$7_2$. In Appendix~\ref{appendix:data}, we have listed the twists knots and their types.

\begin{figure}[htbp]
\begin{center}
\begin{overpic}[scale=1]{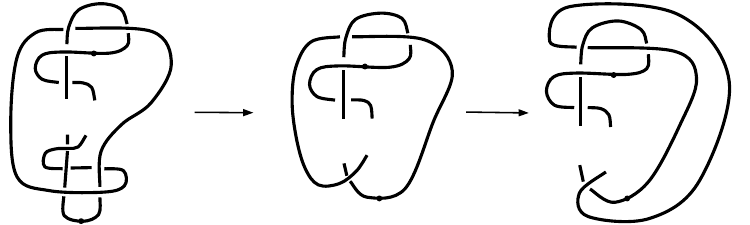}
\put(11.75,23.75){$a$}
\put(9.25,20.25){$L_i$}
\put(5, 16.5){$R_o$}
\put(6, 11){$L_i$}
\put(5.25,5.25){\small{$R_o$}}
\put(5.5, 2){$L_i$}
\put(5.5,27){$C_o$}
\put(17.25,27){$t_1$}
\put(13.75, 8.5){$t_2$}
\put(13.75, 2){$t_3$}
\put(10,13){$\vdots$}
\put(9.5,-2.25){$b$}
\put(48,22){$a$}
\put(46.25,18.25){$L_i$}
\put(42,14){$R_o$}
\put(43.25,6.5){$L_i$}
\put(42.5,26){$C_o$}
\put(54.5,25.75){$t_1$}
\put(47,10){$\vdots$}
\put(49,1){$b$}
\put(81,20.85){$a$}
\put(78,17.5){$L_i$}
\put(73.5,13){$R_o$}
\put(74.5,5){$L_i$}
\put(74,24.75){$C_o$}
\put(86.5, 24.5){$t_1$}
\put(79,9){$\vdots$}
\put(83.75,1.25){$b$}
\end{overpic}
\caption{ On the left, a tie knot that after reduction I gives a type 1 twist knot. The latter is shown in the middle and right.}
\label{fig:TwistKnot-v1}
\end{center}
\end{figure}
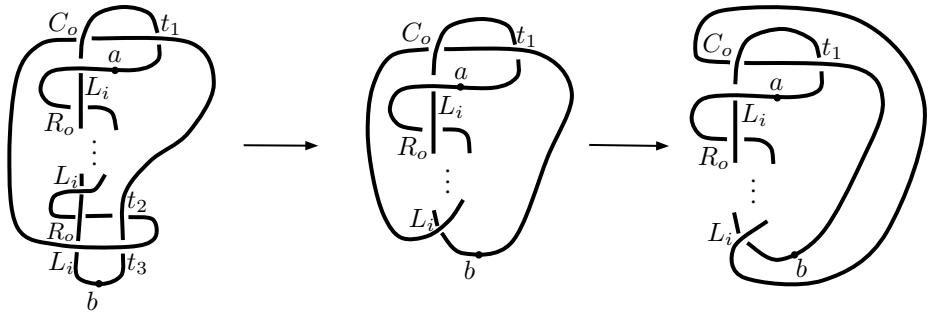

Now assume a tie sequence reduces to type 1. Then the $g(L,R)$ gives the $n$ twists, and the $C_o$ and $\ast$ crossings form the clasp. This is illustrated in \fig{TwistKnot-v1}.
For type 2, the first $L$ and $R$ crossings form the clasp, and the rest of the sequence gives the $n$ twists. This is illustrated in \fig{TwistKnot-v2}. Moreover, by construction there is an even number of twists.
 For type 3, the first $L$ and $C$ crossings form the clasp, and the rest of the sequence gives the $n$ twists. This is illustrated in \fig{TwistKnot-v3}. Again we can see that by construction there is an odd number of twists.
\end{proof}

\begin{figure}[htbp]
\begin{center}
\begin{overpic}[scale=1]{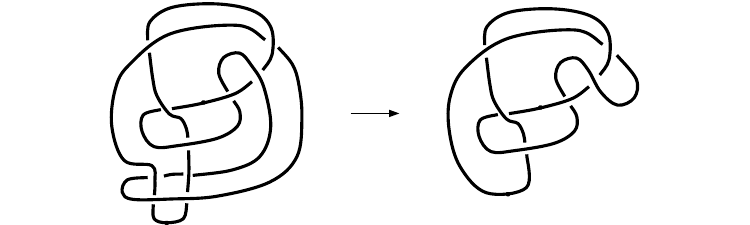}
\put(26,17){$a$}
\put(18,16){$L_o$}
\put(21.5, 11.5){$R_i$}
\put(28, 14){$C_o$}
\put(32.3, 16){$R_i$}
\put(25.5, 7.75){$L_o$}
\put(25, 0.75){$R_i$}
\put(36.75, 24){$C_o$}
\put(17.5, 24.25){$t_1$}
\put(20.75, 7.5){\small$t_2$}
\put(18, 0.75){$t_3$}
\put(21.5, -2.25){$b$}
%
\put(70,16){$a$}
\put(63,15.25){$L_o$}
\put(66.75, 10.75){$R_i$}
\put(73, 13.5){$C_o$}
\put(77.5, 15.5){$R_i$}
\put(81.5, 23.25){$C_o$}
\put(62.4, 23.85){$t_1$}
\put(67, 1.25){$b$}
\end{overpic}
\caption{On the left, tie knot $L_oR_iC_oR_iL_oR_iC_oT $, which reduces to $L_oR_iC_oR_iC_ot_1$ on the right. This is a  type 2 twist knot. }
\label{fig:TwistKnot-v2}
\end{center}
\end{figure}
\begin{figure}[htbp]
\begin{center}
\begin{overpic}[scale=0.9]{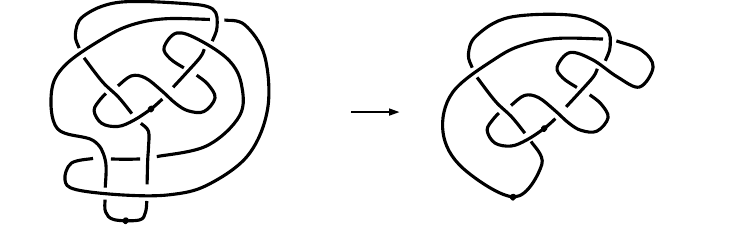}
\put(18,16){$a$}
\put(16.5, 10.5){$L_i$}
\put(13, 20){\small$C_o$}
\put(20.75, 13.75){$R_i$}
\put(24, 17){$C_o$}
\put(27, 20.2){$R_i$}
\put(20, 6){$L_o$}
\put(20, 1){$R_i$}
\put(28.75, 28.25){$C_o$}
\put(7.5, 23.25){$t_1$}
\put(11.5, 6.25){\small$t_2$}
\put(11.25, 1.5){$t_3$}
\put(16, -2.5){$b$}
%
\put(70.5,13.5){$a$}
\put(68.75, 8){$L_i$}
\put(65.5, 17.25){\small$C_o$}
\put(73.25, 11){$R_i$}
\put(76.5, 14.5){$C_o$}
\put(79.25, 17.25){$R_i$}
\put(81.5, 25){$C_o$}
\put(60.25, 20.5){$t_1$}
\put(67, 0.75){$b$}
\end{overpic}
\caption{On the left, tie knot $L_iC_oR_iC_oR_iL_oR_iC_oT $, which reduces to $L_iC_oR_iC_oR_iC_ot_1$. This is a type 3 twist knot.}
\label{fig:TwistKnot-v3}
\end{center}
\end{figure}
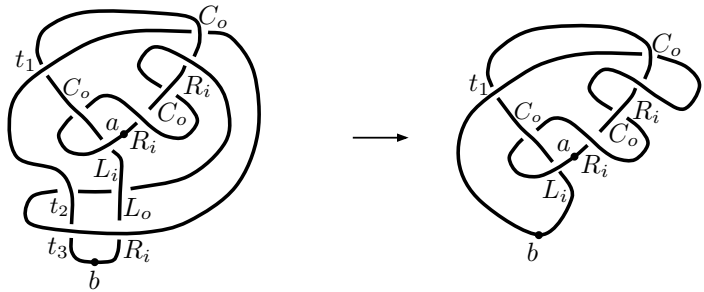

We can again use \lem{reduction} to understand the number of tie sequences that appear in Table~\ref{tab:table2}. Once the twist knot has been established at the start of the tie sequence, the rest of the sequence is reduced in two ways. The first way is to just use reduction I, the second way uses reduction II (and sometimes~IV). The second way automatically gives 4 tie sequences: two for the choice of way to tying the flat facade (taking up the last four spots), and two for the choice of order of sequence of alternating $L$ and $R$ just before that. Thus the second way requires at least an extra 4 moves after the twist knot is established, so we only see this for $3_1, 4_1$ and $5_1$ twist knots.

Let's go through Table~\ref{tab:table2} carefully. We have already discussed the trefoil $3_1$. However, note that the single tie sequence with 4 tie moves gives a type 1 twist knot, and all the other longer sequences use reductions II and IV to give a type 3 twist knot. Similarly, for $4_1$, the single tie sequence with 5 tie moves gives a type 1, and the rest use reductions II and IV to give a type 2 twist knot. For $5_2$, there are two tie sequences with 6 tie moves: one of type 1 and one  of type 3. The other tie sequences with 9 tie moves use reduction II  to give four of type 1 and four of type 3 twist knots. For $6_1, 7_2$ and $8_1$, the only reduction used is reduction I. Thus for $6_1$ and $8_1$  there is one tie sequence of type 1 and one sequence of type 2. For $7_2$ there is one tie sequence of type 1 and one of type 3. 
 
Examples of tie knots which are twist knots include: Kelvin, Hanover, Victoria, and Balthus tie knots. The Kelvin and Hanover  are figure-eight knots, while the Victoria and Balthus are both $5_2$ knots.

\subsection{$(2,p)$ torus knots}

Recall that the $(2,p)$ torus knots (with $p$ odd) are the knots which can be embedded on a torus, and which wrap twice around the longitude and $p$ times around the meridian. The $(2,5)$ torus knot is illustrated in \fig{exTwistTorus}. Note that in any standard knot table the $(2,p)$ torus knots are traditionally given the label $p_1$. For example, the $(2,5)$ torus knot is $5_1$.

\begin{lemma}\label{lem:torus}
If a tie sequence reduces to $LC(LCLC\dots LC)\ast$, then the corresponding tie knot is a $(2,p)$ torus knot. Here, $p$ is the number of tie moves in the reduced tie sequence, and $\ast$ is $L$, $R$, or $t_1$. In addition, all of the FM-tie $(2,p)$ torus knots have a reduced tie sequence of the form $LC(LCLC\dots LC)\ast$.
\end{lemma}
\begin{proof} The sequence $LC(LCLC\dots LC)\ast$ gives an alternating knot diagram with $p$ twists occurring between the active strand and the neck loop (as shown in \fig{torus-examples}). The tie construction means that the resulting knot is a $(2,p)$ torus knot.  An examination of the tie sequences for $3_1$, $5_1$ and $7_1$ knots in Appendix~\ref{appendix:data}  shows that this is the only reduced tie sequence that appears.  
\end{proof}

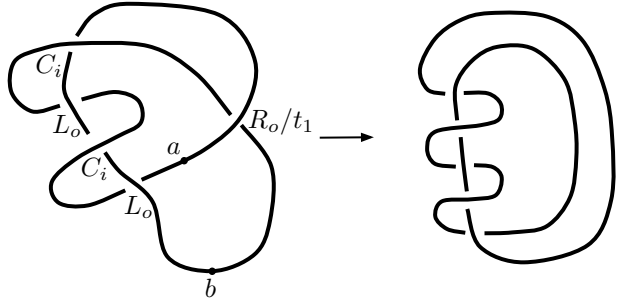
\begin{figure}[htbp]
\begin{center}
\begin{overpic}[scale=1]{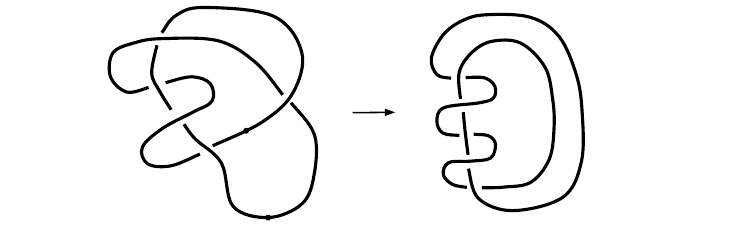}
\put(31,13.25){$a$}
\put(26.5,7){$L_o$}
\put(22,11){$C_i$}
\put(19,15.5){$L_o$}
\put(17.25,21.75){$C_i$}
\put(39.5,16){$R_o/t_1$}
\put(35,-1.75){$b$}
\end{overpic}
\end{center}
\caption{The sequences $L_oC_iL_oC_iR_oC_iL_oR_iC_oT$ and $L_iC_oL_iR_oL_iC_oT$ both reduce to $LCLC\ast$, a $(2,5)$ torus knot.}
\label{fig:torus-examples}
\end{figure}

The data in Table~\ref{tab:table2} is now clear. As with the other knot types, the number of tie sequences that appears depends on whether reduction I is used, or whether reductions II and IV are needed. The $3_1$ knot has been discussed extensively above. For $5_1$, there is one tie sequence with 6 tie moves, and 4 tie sequences with 9 tie moves. For $7_1$ we see there is just the (expected) single tie sequence.

We observe that none of the commonly known tie knots are $(2,p)$ tours knots. Instead, we list the remaining familiar named tie knots with their knot types: Cavendish $7_4$, Christensen $7_7$, and Granchester $8_4$. 

The 85 FM-tie knots were an inspiration for the results in this section. However, as the number of tie moves increases, it is entirely possible that there may be other tie sequences that appear for the knot types discussed in this section.  Such an example might be found for the twist knot family. We conjecture that Lemmas~\ref{lem:unknot}, \ref{lem:trefoil}, and \ref{lem:torus} give the only possible tie sequences for the unknot, trefoil and $(2,p)$ torus knots. 

\section{Further properties of tie knots}\label{section:family}

In the study of knot theory, it is common to study properties of families of knots. For example, properties of alternating knots, twists knots, and torus knots. In this paper we have introduced the family of tie knots, and in this section we will deduce further properties of these knots.

Many properties of this family follow immediately from the work in the previous sections.  If we take a close look at the tie knots of the $0_1, 3_1, 4_1, 5_1$ and $5_2$ knots, we notice that once a knot appears as a tie knot, it will appear infinitely often as the number of tie moves increases. This can be generalized to any tie knot.

\begin{proposition}\label{prop:infinitely-often}
 Suppose a knot appears as a tie knot with $n$ tie moves.
 \begin{enumerate}
 \item If the tie sequence ends in $L_i R_o L_i C_o T$, or $R_i L_o R_i C_o T$, then this knot appears again as a tie knot with $n+3$, $n+5$, $n+7$, \dots tie moves. 
 \item If the tie sequence ends in $C_i R_o L_i C_o T$, or $C_i L_o R_i C_o T$, then this knot appears again as a tie knot with $n+2$, $n+4$, $n+6$, \dots tie moves. 
\end{enumerate} 
Moreover, for the tie sequences with $n+2, n+3,$ or more tie moves, there are (at least) four different tie sequences giving the knot type.

\end{proposition}
 
\begin{proof}
For Case 1, we use reduction I to simplify the tie sequence. This gives us a tie diagram with $n-1$ crossings by Corollary~\ref{cor:reduction}, and the reduced tie sequence ends $[\dots]C_o\ast$, where $\ast$ is $L$ or $R$. Change the $\ast$ to a sequence of alternating $L$ and $R$ moves of odd length (starting in either order), then add on either $C_i R_o L_i C_o T$, or $C_i L_o R_i C_o T$ to the end. Altogether, we get a new tie sequence of the original tie knot with $n-1+2k+4=n+3+2k$ tie moves (where $k=0,1,2,\dots$). 

For Case 2, we know that the tie sequence must have a $L_o$ (or $R_o$) before the final 4 moves. Simply add an even number of $R_oL_o$'s after the $L_o$ (or $L_oR_o$'s after the $R_o$) before the final 4 moves. This gives a tie sequence with $n+2$, $n+4$, $n+6$, \dots tie moves. 

In both cases, there are 4 possible tie knots of a particular length: there are two choices of alternating sequences  of $L$ and $R$ (depending on the starting move), and two choices of flat facade. The new tie sequences also have the same knot type as the original knot, as can be seen by using reductions I or II, then III or IV.
\end{proof}

While \prop{infinitely-often} guarantees that tie knot types will appear infinitely often as the number of tie moves increases, it is entirely possible there may be other, different tie sequences for a knot that appear as the number of tie moves increases.  As discussed above, it is possible that the twist knots may have a fourth (or fifth, or more) reduced tie sequence when there are a greater number of tie moves.

We also observed that the left and right trefoil knots appeared in our list of tie knots. These knots are mirror images\footnote{Recall that to find the mirror image of a knot, we take a knot diagram, then switch all the crossings.} of each other.  We have not tried to find the mirror images of any other knot types as a tie knot. However, arguments similar to \prop{infinitely-often} allows us to immediately prove the following.

\begin{proposition} If a knot is a tie knot, its mirror image is also a tie knot.
\end{proposition}
\begin{proof}  Assume the tie sequence of the knot has $n$ tie-moves. First, perform either reduction I or II (which ever is needed),  and create the corresponding tie diagram of the reduced tie sequence. Second, switch all the crossings by changing in-moves to out-moves and vice-versa. This creates a knot diagram of the mirror image of the original knot.  Finally, we add tie moves to the reduced tie sequence to create a tie knot that represents the mirror image of the original knot. To do this, we consider 2 cases. (Recall that tie sequences starting $L_o$ have an odd number of tie moves, and those starting with $L_i$ have an even number of tie moves.)

Case~1a: We perform reduction I and $n$ is even. This gives a diagram of $n-1$ (odd) tie moves with sequence $L_i[\dots]C_ot_1$. When we switch the crossings, the first move becomes $L_o$ and the $t_1$ becomes $L_o$ or $R_o$. We then add either $C_i R_o L_i C_o T$, or $C_i L_o R_i C_o T$ to the end.

Case~1b: We perform reduction I and $n$ is odd. We proceed in a similar way to Case 1a. 

Case 2a:  We perform reduction II and $n$ is even. This gives a diagram of $n-4$ (even) tie moves ending in $L_i$ or $R_i$. When we switch the crossings the first move becomes $L_o$ (requiring an odd number of tie moves). We thus add an extra $R_o$ after $L_i$, or $L_o$ after $R_i$, before adding $C_i R_o L_i C_o T$, or $C_i L_o R_i C_o T$ to the end.

Case 2b: We perform reduction II and $n$ is odd. We proceed in a similar way to Case 2a.

\noindent
In all cases the final tie sequence gives a knot equivalent to the mirror image, as can be seen by using reduction II, then III or IV.
\end{proof}
\subsection{Tie Knots Alternate}\label{alternating}

We can now ask a very simple question: are all knots tie knots? We prove in \thm{alternate} below, that all tie knots are alternating. Thus we can definitely say that not all knots are tie knots, since not all knots are alternating. Indeed, in 2018,  H.~Chapman~(\cite{chapman} Theorem~1) proved that all but exponentially few knot types (prime or composite) are nonalternating.  More recently, Y. Belousov and A. Malyutin~\cite{BM} proved that hyperbolic knots are not generic, which implies Chapman's main result.

\begin{theorem}\label{thm:alternate}
All tie knots are alternating.
\end{theorem}

\begin{proof} We prove that such a tie knot has an alternating tie diagram.  Recall, we view the tie  diagram as having two parts: the active strand from start point $a$ to $b$, and passive strand of the knot following the straight strand from point $b$, around the neck loop and finishing at point $a$. \fig{alternating-proof} right illustrates this, where we have deliberately not shown whether the first crossing is an over- or under-crossing.

\begin{figure}[htbp]
\begin{center}
\begin{overpic}[scale=1]{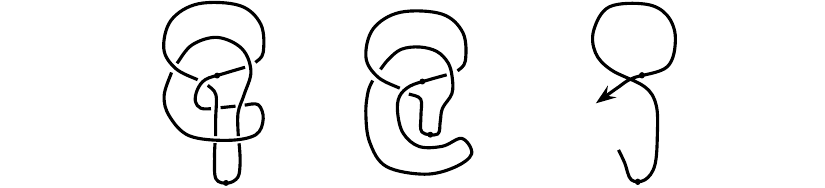}
\put(21,11.5){$L_i$}
\put(22.75,7.3){$R_o$}
\put(23,3.5){$L_i$}
\put(17.25,13.5){$C_o$}
\put(25.5,14.3){$a$}
\put(27,-2){$b$}
\put(31,13){$t_1$}
\put(29.3,7.7){$t_2$}
\put(29.75,3.7){$t_3$}
\put(45.5,10.25){$L_i$}
\put(42,13){$C_o$}
\put(50.25,13.5){$a$}
\put(51.75,7){$b$}
\put(55.5,12.25){$t_1$}
\put(77,14.5){$a$}
\put(77,-2){$b$}
\put(83,17){neck loop}
\put(81,7){straight strand}
\put(63,8){active strand}
\put(72,5){$\ddots$}
\end{overpic}  
\caption{The four-in-hand tie knot (left) has an alternating diagram (center). On the right, the passive strand consists of the neck loop and the straight strand, while the active strand (not shown) starts at point $a$ and follows on from the arrow to point $b$.}
\label{fig:alternating-proof}
\end{center}
\end{figure}
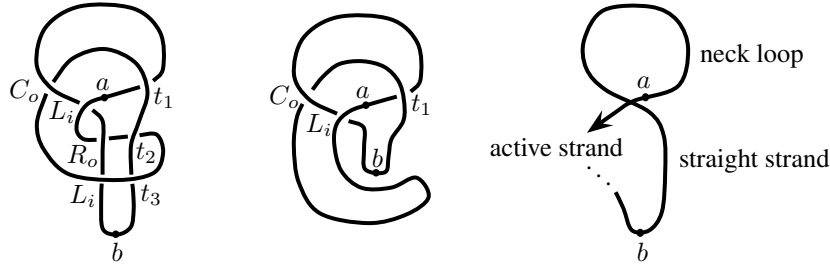

\cor{reduction} tells us that, after applying reductions 0--IV, our tie sequence is $\emptyset$ or $L[\dots]C\ast$, where $\ast$ is a $L$,  or $R$ move. The $\emptyset$ tie sequence corresponds to the unknot, which is alternating. We claim the tie diagram corresponding to the $L[\dots]C\ast$ tie sequence is also alternating.  As we move along the active strand, we know that the crossings corresponding to $L$, $R$, and $C$ moves alternate by FM-assumption (2). Since the reductions have either removed the tuck move entirely, or have reduced it to $t_1$ ($L$ or $R$), we know that as we move along the active strand, all the crossings in the corresponding tie diagram alternate between over and under (or vice-versa).  What remains is to check the crossings alternate as we travel along the passive strand.

For the rest of the proof, we will assume that the first tie move in the tie-sequence is $L_i$. (Similar arguments will hold when the first move is $L_o$.) In this case, the active strand crosses over the passive strand just to the left of point $a$ forming the neck loop. 
We now consider crossings on the straight strand moving down from the first crossing $L_i$. The first crossing down from $L_i$ must be an out-move, an under-crossing. This is because, by FM-assumptions~(1) and (2), the active strand must cross the neck loop 0, 2, or an even number of times before returning to the straight strand. (For example, the tie sequence could start $L_iR_o$, $L_iC_oL_iR_o$, or $L_iC_oR_iL_0$. In each of these cases the last move is an out-move and crosses under the straight strand after the first $L_i$.) The same reasoning shows the second crossing down the straight strand from $L_i$ will be an in-move (over-crossing), the third crossing will be an out-move (under-crossing), and so on. 

Let us now consider the neck loop formed by the first crossing $L_i$. This bounds a region isotopic to a disk in the plane of the tie diagram.  We know that after the first crossing $L_i$, the active strand is outside of this region. We also know our reduced tie sequence ends $C\ast$. Altogether this means that anytime the active strand enters the region, it must leave it again, and hence there are an even number of crossings along the neck loop (not including the  $L_i$ crossing creating the loop).

We claim the crossings must alternate as we travel counter-clockwise along the neck loop from point $a$. Again using FM-assumptions (1) and (2), there are two cases starting at the first $L_i$ crossing. 

{\bf Case A:}  the active strand crosses the straight strand 1, 3, or an odd number of times, entering the central region from the right side, with tie sequence $L_i(R_oL_i\dots R_oL_i)R_oC_i\ast_o$. (Here, the ($R_iL_o\dots R_iL_o$) is the sequence of moves along the straight strand, and $\ast$ is either $L$ or $R$.)  If we walk counter-clockwise from $a$ along the neck loop, we see $C_i$ (coming from $R$ region), then $\ast_o$, then $L_i$.  Thus, the crossings alternate along the neck loop. In particular,  there are two $_i$-moves; one before and one after point $a$.  

{\bf Case B:}  the active strand crosses the straight strand 0, 2, or an even number of times, and enters the central region on the left side, with tie sequence $L_i(R_oL_i\dots R_oL_i)C_o\ast_i$.  If we again walk counter-clockwise from $a$ along the neck loop, we see $\ast_i$, then $C_o$ (coming from $L$ region), then $L_i$. Again, the crossings alternate, and there are in-moves on either side of point $a$.  These in-moves are ``in" from the perspective of the active strand. Thus, as we move from the active strand to the passive strand at point $a$, we see that the crossings do indeed alternate.

We next claim that any other pairs of crossings along the neck loop will preserve the alternating pattern.  Let's consider Case A above, and assume the tie sequence is $L_i(R_oL_i\dots R_oL_i)R_oC_iL_o$. We will see that the next the pair of crossings into and out of the central region occur on the neck loop between the final $C_i$ and $L_o$. The next crossing into the central region could happen on the left side (after the active strand crosses the straight strand an even number of times), with sequence ending $\dots R_oC_iL_o(R_iL_o\dots R_iL_o)C_i\ast_o$.   Alternatively, the next crossing into the central region could happen on  the right side (after the active strand crosses the straight strand an odd number of times), with sequence ending $\dots R_oC_iL_o(R_iL_o\dots R_iL_o)R_iC_o\ast_i$.  With either option, as we walk counterclockwise along the neck loop from point  $a$, we see all the crossings alternate by construction. If we now assume the sequence from Case A ends $L_i(R_oL_i\dots R_oL_i)R_oC_iR_o$, a similar argument holds showing the crossing alternate traveling along the neck loop from point $a$. The main difference is that the next pair of crossings along the neck loop occurs between the last $R_o$ and the initial $L_i$.   We can also mimic this argument for Case B from the previous paragraph. Exactly where the new pair of crossings appears depends on whether the  active strand crosses into the central region from the left or right side. Regardless, the crossings around the neck loop are alternating, and there are in-moves on either side of point $a$. The same arguments also hold for the next pair of crossings along the neck loop, and the next, and so on.

It remains to check that the crossings on either side of $b$ are alternating. We will again use the fact that our reduced tie sequence ends $C\ast$, and the active strand joins the passive strand at point $b$ after this last $\ast$ crossing. There will be a final crossing on the straight strand before the active loop crosses the neck loop for the final (even) number of crossings ending $C\ast$. We first assume there are an odd number of crossings on the straight strand, where we do include the first $L_i$ crossing in our count. This means the last crossing on the straight strand is also an in-move. We showed above that that the active strand crosses the neck loop an even number of times, thus the last crossing out of the neck loop is also an in-move. This means that the crossings on either side of point $b$ are both in-moves. We next assume there are an even number of crossings on the straight strand, so that the last crossing on the straight strand is an out-move. The active strand crosses the neck loop an even number of times, thus the last crossing out of the neck loop is also an out-move. Again, the crossings on either side of point $b$ are both out-moves. In both cases the crossings either side of $b$ are from the perspective of the active strand, and we observe that the crossings do indeed alternate on either side of point $b$.
\end{proof}

As discussed above, we can immediately deduce the following.

\begin{corollary}
Not all knots are tie knots.
\end{corollary}
\begin{proof}
By Theorem~\ref{thm:alternate}, all tie knots are alternating. However, not all knots are alternating. Therefore, not all knots are tie knots.
\end{proof}
\subsection{Other properties of tie knots}

There are several other natural questions which can be asked about the family of tie knots. The FM assumptions (0)--(3) are quite restrictive, and allows us to make several more deductions about tie knots. Before we give the first, recall that a  {\em nugatory crossing} in a knot diagram is one that can be removed by twisting. More formally, a nugatory crossing is one for which there exists a topological circle in the plane of the knot diagram meeting the crossing transversely, but not meeting the knot diagram anywhere else.  A knot diagram is {\em reduced} if it does not contain any nugatory crossings.

\begin{proposition}\label{prop:reduced-diagram} Any tie knot has a reduced tie digram. 
\end{proposition}
\begin{proof} Take the tie sequence for the tie knot, and use reductions 0--IV to simplify the corresponding tie diagram. \cor{reduction} tells us reduced tie sequence is either $\emptyset$ or $L[\dots]C\ast$, where $\ast$ is either $L$ or $R$. The $\emptyset$ tie sequence corresponds to the unknot, which has a reduced tie diagram (the round circle). We will show that $L[\dots]C\ast$ tie sequence is also reduced by showing any (topological) circle intersecting a crossing transversely must meet the knot diagram again. Recall that the crossings in the reduced tie diagram only occur along the neck loop and the straight strand.

Consider any crossing along the neck loop. The neck loop bounds a region isotopic to a disk in the plane of the knot diagram. This means any closed curve entering the neck loop must leave it again. In particular, any circle transversely intersecting a crossing along the neck loop must meet the neck loop (part of the knot diagram) again.   

Consider any crossing along the straight strand, including the first crossing. Since our reduced tie sequence ends $C\ast$, we know the straight stand must eventually turn and cross the neck loop at crossing $\ast$. This creates another loop, which we call the {\em straight loop}, again bounding a region isotopic to a disk. To be clear, the straight loop consists of the straight stand starting at the crossing of the first tie move, then going past the point $b$, then up to the point $\ast$, then along the neck loop back to the first tie move. This is shown in Figure~\ref{fig:reduced-prime} right. Again, any circle transversely intersecting a crossing along the straight strand must meet the straight loop (part of the knot diagram) again. We have shown the tie diagram is reduced. 
\end{proof}

\begin{figure}[htbp]
\begin{center}
\begin{overpic}[scale=1]{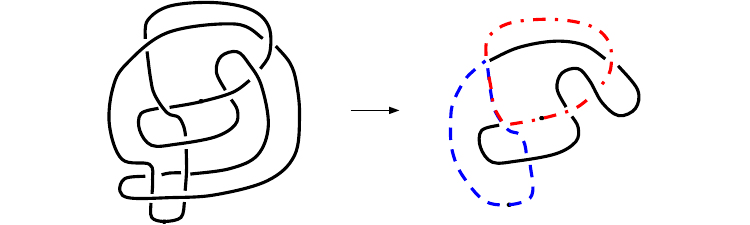}
\put(26,17){$a$}
\put(17.8,16.2){$L_o$}
\put(21.5, 11.65){$R_i$}
\put(28, 14.5){$C_o$}
\put(35, 18.5){$R_i$}
\put(25.5, 7.75){$L_o$}
\put(24.6, 1){$R_i$}
\put(36.75, 24){$C_o$}
\put(17, 24.25){$t_1$}
\put(20.75, 7.65){\small$t_2$}
\put(17.5, 1){$t_3$}
\put(21.25, -2.25){$b$}
%
\put(70.5,14.75){$a$}
\put(63,14){$L_o$}
\put(66.75, 9.75){$R_i$}
\put(73.5, 12){$C_o$}
\put(80.5, 16.5){$R_i$}
\put(82, 22){$C_o$}
\put(56, 22.85){$L_i=t_1$}
\put(67.25, 0){$b$}
\end{overpic}
\caption{On the left, tie knot $L_oR_iC_oR_iL_oR_iC_oT $, which reduces to $L_oR_iC_oR_iC_oL_i$ on the right. The neck loop is shown in red (dot-dash line) and the straight loop is shown in blue (dashed line). These loops overlap along the neck loop between the $L_o$ and the $t_1=L_i$ crossings.}
\label{fig:reduced-prime}
\end{center}
\end{figure}
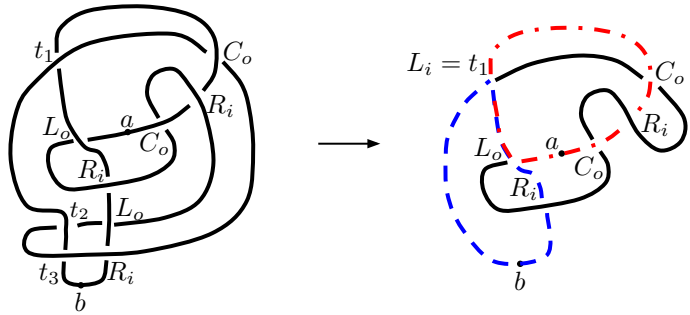

By combining \thm{alternate} and \prop{reduced-diagram}, we have proved that:
\begin{corollary} Any tie knot has a reduced, alternating knot diagram. 
\end{corollary} 
Kauffman, Murasugi and Thistlethwaite proved \cite{adams, liv, mura} that an alternating knot in a reduced alternating projection of $n$ crossings has crossing number $n$.  This allows us to  conclude that:

\begin{corollary}\label{cor:min-crossing} Any tie knot has a minimal crossing diagram corresponding to its reduced tie sequence.
\end{corollary}

Recall that a prime knot is a knot which is not composite. A composite knot $K\#J$ has a diagram where there is a topological circle which intersects the diagram in just two arcs. The circle separates the diagram into sub-knots $K$ and $J$, where one is on the inside of the circle and one is on the outside.

\begin{proposition} All tie knots are prime knots.
\end{proposition}
\begin{proof} 
Given the tie knot, use reductions 0--IV to create a reduced tie diagram denoted by $K$. If the tie knot is the unknot, it is prime.  We will show that any (topological) circle that contains some, but not all, of the crossings of $K$ must intersect $K$ in three or more crossings. This means that all such circles are not separating, and so the tie knot must be prime. 

Recall that any crossing in $K$ only occurs when the active strand intersects either the neck loop or the straight strand. As shown in Figure~\ref{fig:reduced-prime}, there are two topological circles of note in $K$:  the neck loop and the straight loop. We note the  {\em first crossing} of $K$ is the crossing which creates the neck loop. Now, suppose a (topological) circle $C$ contains one (or more) crossings of $K$.   Since we are trying to show the knot is prime, we assume $C$ does not contain all of the crossings of $K$.  There are four main cases to consider. 

Case A: Assume that $C$ contains at least one, but not all of the crossings along the neck loop. Then, because the neck loop is also a topological circle, there is (at least) one arc of the neck loop inside~$C$ which intersects $C$ (at least) twice.  We know the second arc creating the crossing inside $C$ comes from the active strand. Start at the crossing inside $C$ and follow this active strand. This strand must leave $C$ (at least) once to get to the crossing(s) outside of $C$. Thus $C$ has (at least) three intersections with $K$ and so is not a separating circle.

Case B: Assume that $C$ contains all of the crossings along the neck loop including the first crossing of $K$. (In this case, $C$ may also contain some, but not all crossings along the straight strand.)  Then, since the straight loop is a topological circle sharing an arc of the neck loop, we see that $C$ intersects the straight loop in (at least) two places. By assumption, there is at least one crossing of $K$ along the straight strand outside of $C$. Thus, the active strand must cross $C$ at least once (from the first crossing to the crossing outside $C$). Thus $K$ crosses $C$ in (at least) three places, and $C$ is not a separating circle.

Case C: Assume that $C$ contains at least one crossing on the straight strand, but does not contain the first crossing of $K$. This crossing also lies on the straight loop, a topological circle.  This means there is (at least) one arc of the straight strand inside $C$ and this arc intersects $C$ (at least) twice.  The second arc creating the crossing inside $C$ comes from the active strand. The active strand starts at the first crossing, which is outside of $C$. Thus the active strand crosses $C$ (at least) once. Altogether we see $C$ has (at least) three intersections with $K$, and is not a separating circle.

Case D: Assume $C$ contains crossings from both the straight strand and the neck loop (so includes the first crossing, and possibly other crossings along the neck loop). We also assume $C$ does not contain all the crossings from the neck loop, or else we would be in Case B. Then $C$ contains (at least) one arc of the neck loop, which intersects $C$ (at least) twice. Since $C$ contains the first crossing, and the active strand starts here, we know the active strand must intersect $C$ to get to the crossing on the neck loop outside of $C$. Thus $C$ has (at least) three intersections with $K$, and is not a separating circle.
\end{proof}

\subsection{Open questions}
We end with a conjecture and an open question. If we look at a knot table containing all knots with up to 8 crossings, then the only knots which do not appear as tie knots are $8_5$, $8_{10}$, and then $8_{15}$ through $8_{21}$. Of these, only $8_{19}, 8_{20}$ and $8_{21}$ are non-alternating. 

Recall that the bridge number of a knot is the minimum number of bridges required for any diagram of the knot. (A bridge is an arc of a knot diagram which includes at least one overcrossing.) Interestingly\footnote{Charles Livingston's KnotInfo: Table of Knot Invariants~\cite{liv-table} is a huge help here.}, the only knots with up to 8 crossings that are 3-bridge knots are $8_5$, $8_{10}$, then $8_{15}$ through $8_{21}$. These are precisely the knots which do not appear as tie knots!  All the other tie knots are 2-bridge knots. For higher crossing numbers, there can be both alternating and non-alternating knots of any bridge index. Altogether, this leads us to conjecture.

\begin{conjecture} All tie knots are prime, alternating 2-bridge knots.
\end{conjecture}

We end with another open question about tie knots. When considering a tie knot as an actual physical tie, the tie can be modeled as a strip of paper with the knot as the core curve. As such, the tie knot is a framed knot. Can we work out the self-linking number of a tie knot? In their paper \cite{finkmao}, Fink and Mao describe two additional characteristics of the shape of the tie knot: symmetry and balance. The latter is the number of reversals of the winding direction of the tie. Perhaps this quantity could be used in a calculation of linking number.

\section{Acknowledgements}
The authors wish to thank the referee for their suggestions and corrections. The exposition is clearer as a result of this advice.

The authors are grateful for the support of Washington \& Lee University. The first author's research was funded by two Washington \& Lee Lenfest research grants for summers 2018 and 2019. The second and third authors' research was funded by the Washington \& Lee University Summer Research Scholars Program 2018. All of the authors wish to thank Professor Wayne Dymaceck for donating ties for us to use, and Professor Aaron Abrams for lending us knot toys and thin ropes. The ties and the toys made studying tie knots much easier.
\appendix 
\section{Classification of The Knots}
\label{appendix:data}
 This is the table with our data of all 85 tie sequences and their respective knot types. The table gives the number of the tie given by Fink and Mao \cite{finkmao}, then the number of tie moves, then the tie sequence, then the knot type of the corresponding tie knot. The next column gives the type of tie sequence for twist knots given in \lem{twist}, and the final column gives the common name for the tie knot.
 
 \begin{table}
\begin{tabular}{| c | c | l | c | c |  c|}
\hline
FM-number & \# tie moves & tie sequence & knot type & twist knot type & common name \\
\hline
\hline
1& 3 &	$L_oR_iC_oT$	& $0_1$ & &Oriental \\
4 &	5 &	$L_oC_iR_oL_iC_oT$ &	$0_1$ & & Nicky \\
5 &	5 &	$L_oC_iL_oR_iC_oT$ &	$0_1$ & & Pratt/Shelby \\
7 &	6 &	$L_iR_oC_iL_oR_iC_oT$ & $0_1$ & & half-Windsor\\
8 &	6 &	$L_iR_oC_iR_oL_iC_oT$ & $0_1$ & & \\
12 & 	7 & 	$L_oR_iL_oC_iR_oL_iC_oT$ & $0_1$ & & St Andrew \\
14 & 	7 & 	$L_oR_iL_oC_iL_oR_iC_oT$ & $0_1$ & &\\
24 & 	8 &	$L_iR_oL_iR_oC_iL_oR_iC_oT$ & $0_1$ & & \\
26 & 	8 &  	$L_iR_oL_iR_oC_iR_oL_iC_oT$ & $0_1$ & & \\
46 & 	9 & 	$L_oR_iL_oR_iL_oC_iR_oL_iC_oT$ & $0_1$ & & \\
50 &	9 &	$L_oR_iL_oR_iL_oC_iL_oR_iC_oT$	 & $0_1$  & & \\
\hline
2 &	4 & 	$L_iR_oL_iC_oT$ & 	$3_1$ & 1 & four-in-hand\\
18 & 	7 & 	$L_oC_iR_oC_iL_oR_iC_oT$  & $3_1$ & 3 & Plattsburgh \\
19 &	7 & 	$L_oC_iR_oC_iR_oL_iC_oT$ 	& $3_1$ & 3 & \\
20 & 	7 & 	$L_oC_iL_oC_iR_oL_iC_oT$ 	& $3_1$ & 3 & \\
21 & 	7 & 	$L_oC_iL_oC_iL_oR_iC_oT$ & $3_1$ & 3& \\
31 &	8 &	$L_iC_oR_iL_oC_iR_oL_iC_oT$ 	& $3_1$ & 3 & Windsor \\
32 &	8 & 	$L_iC_oL_iR_oC_iL_oR_iC_oT$  	& $3_1$ & 3 & co-Windsor or Persian\\
33 &	8 & 	$L_iC_oR_iL_oC_iL_oR_iC_oT$ 	& $3_1$ & 3 & \\
35 &	8 & 	$L_iC_oL_iR_oC_iR_oL_iC_oT$ 	& $3_1$ & 3 & \\
57 &	9 &	$L_oC_iR_oL_iR_oC_iL_oR_iC_oT$  & $3_1$ & 3 & \\
60 &	9 &	$L_oC_iR_oL_iR_oC_iR_oL_iC_oT$ 	& $3_1$ & 3 & \\
68 &	9 & 	$L_oC_iL_oR_iL_oC_iR_oL_iC_oT$ 	& $3_1$ & 3& \\
71 & 	9 & 	$L_oC_iL_oR_iL_oC_iL_oR_iC_oT$  & $3_1$ & 3 & \\
\hline
3 & 	5 & 	$L_oR_iL_oR_iC_oT$  &  $4_1$ & 1 & Kelvin \\
34 & 	8 & 	$L_iR_oC_iL_oC_iR_oL_iC_oT$  	 &  $4_1$ & 2 & \\
36 & 	8 & 	$L_iR_oC_iR_oC_iL_oR_iC_oT$   &  $4_1$  & 2 & \\
37 &	8 &	$L_iR_oC_iL_oC_iL_oR_iC_oT$   &  $4_1$ & 2 & \\
39 & 	8 & 	$L_iR_oC_iR_oC_iR_oL_iC_oT$  	 &  $4_1$  & 2 & \\
54 &	9 & 	$L_oR_iC_oL_iR_oC_iL_oR_iC_oT$  	&  $4_1$ & 2 & Hanover \\
55 & 	9 & 	$L_oR_iC_oR_iL_oC_iR_oL_iC_oT$  	&  $4_1$ & 2 & \\
56 &	9 &	$L_oR_iC_oL_iR_oC_iR_oL_iC_oT$  	&  $4_1$ & 2 & \\
59 &	9 & 	$L_oR_iC_oR_iL_oC_iL_oR_iC_oT$  	&  $4_1$  & 2 & \\
\hline
10 & 	6 &	$L_iC_oL_iR_oL_iC_oT$	 & $5_1$ &&  \\
79 &	9 &	$L_oC_iL_oC_iR_oC_iL_oR_iC_oT$  & $5_1$ & & \\
82 & 	9 & 	$L_oC_iL_oC_iR_oC_iR_oL_iC_oT$	& $5_1$ & & \\
84 &	9 &	$L_oC_iL_oC_iL_oC_iR_oL_iC_oT$ 	& $5_1$  & & \\
85 &	9  & 	$L_oC_iL_oC_iL_oC_iL_oR_iC_oT$  	& $5_1$ & & \\
\hline
6 & 	6 & 	$L_iR_oL_iR_oL_iC_oT$ 	& $5_2$ & 1 & Victoria \\
9 & 	6 & 	$L_iC_oR_iL_oR_iC_oT$ 	 & $5_2$ & 1 & \\
61 & 	9 & 	$L_oR_iL_oC_iR_oC_iL_oR_iC_oT$ & $5_2$ & 1 & \\
63 & 	9 &	$L_oR_iL_oC_iR_oC_iR_oL_iC_oT$  & $5_2$ & 1 & \\
72 &	9 &	$L_oR_iL_oC_iL_oC_iR_oL_iC_oT$  & $5_2$ & 1 & \\
74 &	9 &	$L_oR_iL_oC_iL_oC_iL_oR_iC_oT$  & $5_2$ & 1 & \\
78 &	9 & 	$L_oC_iR_oC_iL_oC_iR_oL_iC_oT$  & $5_2$ & 3 & Balthus \\
80 & 9 & 	$L_oC_iR_oC_iR_oC_iL_oR_iC_oT$  & $5_2$ & 3 & \\
81 &	9 & 	$L_oC_iR_oC_iL_oC_iL_oR_iC_oT$   & $5_2$ & 3 & \\
83 & 	9 & 	$L_oC_iR_oC_iR_oC_iR_oL_iC_oT$  & $5_2$ & 3 & \\
\hline
11 &	7 &	$L_oR_iL_oR_iL_oR_iC_oT$  	& $6_1$ & 1 & \\
15 & 	7 & 	$L_oR_iC_oR_iL_oR_iC_oT$  	& $6_1$  & 2 & \\
\hline
13 &	7 & 	$L_oR_iC_oL_iR_oL_iC_oT$  & $6_2$ & & \\
17 &	7 &	$L_oC_iL_oR_iL_oR_iC_oT$  	& $6_2$ & & \\
\hline
16 &	7 &	$L_oC_iR_oL_iR_oL_iC_oT$  	& $6_3$ & & \\
\hline 
\end{tabular}
\end{table}

\begin{table}
\begin{tabular}{| c | c | l | c | c| c|}
\hline
FM-number & \# tie moves & tie sequence & knot type & twist knot type & common name \\
\hline
\hline

42 &	8 &	$L_iC_oL_iC_oL_iR_oL_iC_oT$ & $7_1$ & & \\
\hline
22 & 	8 &	$L_iR_oL_iR_oL_iR_oL_iC_oT$  & 	$7_2$ & 1 & \\
41 &	8 &	$L_iC_oR_iC_oR_iL_oR_iC_oT$ 	& 	$7_2$ & 3&  \\
\hline
29 &	8 &	$L_iR_oL_iC_oL_iR_oL_iC_oT$ & $ 7_3$ & & \\
40 &	8 & 	$L_iC_oL_iC_oR_iL_oR_iC_oT$  & $ 7_3$ & & \\
\hline
23& 	8& 	$L_iR_oL_iC_oR_iL_oR_iC_oT$  	& $7_4$ & & Cavendish \\
\hline
30 &	8 & 	$L_iC_oL_iR_oL_iR_oL_iC_oT$  & 	$7_5$ & & \\
38 &	8 &	$L_iC_oR_iC_oL_iR_oL_iC_oT$	  & 	$7_5$ & & \\
\hline
27 &	8 &	$L_iR_oC_iR_oL_iR_oL_iC_oT$  &	$7_6$  & & \\
28 & 	8 & 	$L_iC_oR_iL_oR_iL_oR_iC_oT$  	&	$7_6$ &  & \\
\hline
25 &	8 &	$L_iR_oC_iL_oR_iL_oR_iC_oT$  &  $	7_7$ & & Christensen \\
\hline
43 & 	9 & 	$L_oR_iL_oR_iL_oR_iL_oR_iC_oT$  & $8_1$ & 1 & \\  
75 & 	9 & 	$L_oR_iC_oR_iC_oR_iL_oR_iC_oT$ 	& $8_1$ & 2 &  \\   
\hline
73 & 	9 & 	$L_oR_iC_oL_iC_oL_iR_oL_iC_oT$  &  $	8_2$ & & \\
77 &	9 & 	$L_oC_iL_oC_iL_oR_iL_oR_iC_oT$   &  $8_2$ & & \\
\hline
48 &	9 & 	$L_oR_iL_oR_iC_oR_iL_oR_iC_oT$  & $8_3$ & & \\
\hline
44 &	9 &	$L_oR_iL_oR_iC_oL_iR_oL_iC_oT$  & $8_4$  & & Grantchester\\
65 &	9 & 	$L_oC_iL_oR_iC_oR_iL_oR_iC_oT$  & $8_4$  & &  \\
\hline
53 &	9 &	$L_oC_iL_oR_iL_oR_iL_oR_iC_oT$  & $8_6$  &&  \\
64 &	9 & 	$L_oR_iC_oR_iC_oL_iR_oL_iC_oT$  &  $8_6$ & & \\
\hline
70 & 	9 & 	$L_oC_iR_oL_iC_oL_iR_oL_iC_oT$  & $8_7$ & & \\
76 &	9 & 	$L_oC_iL_oC_iR_oL_iR_oL_iC_oT$  & $8_7$  & & \\
\hline
52 & 	 9 & 	$L_oC_iR_oL_iR_oL_iR_oL_iC_oT$  & $8_8$ & & \\
67 &	9 & 	$L_oC_iR_oC_iR_oL_iR_oL_iC_oT$  & $	8_8$ & & \\
\hline
69 &	9 &	$L_oC_iL_oR_iC_oL_iR_oL_iC_oT$  &  $	8_9$  & &  \\
\hline
49 &	9 & 	$L_oR_iL_oC_iL_oR_iL_oR_iC_oT$  & $8_{11}$ & & \\
62 & 	9 & 	$L_oR_iC_oL_iC_oR_iL_oR_iC_oT$  & $8_{11}$ & &  \\
\hline
51 &	9 &	$L_oR_iC_oR_iL_oR_iL_oR_iC_oT$  & $8_{12}$  &&  \\
\hline
45 & 	9 & 	$L_oR_iL_oC_iR_oL_iR_oL_iC_oT$  &  $	8_{13}$  & & \\
58 & 9 &	$L_oC_iR_oL_iC_oR_iL_oR_iC_oT$  & $8_{13}$ & & \\
\hline
47 &	9 &	$L_oR_iC_oL_iR_oL_iR_oL_iC_oT$  & $8_{14}$  & & \\
66 &	9 & 	$L_oC_iR_oC_iL_oR_iL_oR_iC_oT$  & $8_{14} $ & & \\
\hline
\end{tabular}
\end{table} 
 

\bibliographystyle{amsalpha}

 \end{document}